\documentclass{article}
\usepackage[letterpaper, margin=1in]{geometry}

\usepackage[utf8]{inputenc}
\usepackage{amsmath, amssymb, amsthm,verbatim, hyperref}
\usepackage[dvipsnames]{xcolor}
\usepackage{hyperref}
\usepackage{tikz}
\usetikzlibrary{shapes.symbols}
\usetikzlibrary{shapes.symbols, matrix, arrows}
\usepackage{mathrsfs}
\usepackage[scr=esstix,cal=boondox]{mathalfa}
\usepackage{marginnote}
\usepackage{soul}
\usepackage{enumitem}

\newcommand{\done}{\mathscr{d}_1}
\newcommand{\dtwo}{\mathscr{d}_2}

\setstcolor{purple}
\newcommand{\tzmap}{\theta_{\boldsymbol{z}}}

\newcommand{\Z}{\mathbb{Z}}
\newcommand{\lcm}{\text{lcm}}

\theoremstyle{plain}

\newtheorem{thm}{Theorem}[section]
\newtheorem{prop}[thm]{Proposition}
\newtheorem{cor}[thm]{Corollary}
\newtheorem{lemma}[thm]{Lemma}

\newtheorem{remark}[thm]{Remark}

\theoremstyle{definition}
\newtheorem*{definition}{Definition}

\newtheorem*{example}{Example}
\newtheorem{thm*}{Main Theorem}

\title{Measuring the Space of Metaplectic Whittaker Functions}
\author{Ilani Axelrod-Freed, Claire Frechette, and Veronica Lang}
\date{\today}

\begin{document}

\maketitle




\begin{abstract}
Whittaker functions are special functions that arise in $p$-adic number theory and representation theory. They may be defined on representations of reductive groups as well as their metaplectic covering groups: fascinatingly, many of their number theoretic applications survive the transition between the reductive and metaplectic cases. However, one notable difference is that the space of Whittaker functions on a reductive group over a nonarchimedean local field $F$ is one-dimensional, whereas this is no longer true in the metaplectic case. In a previous paper, the second author showed that the dimension of the space of Whittaker functions on an arbitrary $n$-fold metaplectic cover of $GL_r(F)$ can be counted in terms of the number of solutions to a particular system of linear Diophantine equations in terms of $n$ and $r$. In this paper, we calculate two precise formulae for $\dim(\mathfrak{W})$, one inspired by viewing this system as a homogenous specialization of an inhomogenous system and the other by the structure of the coroot lattice of $GL_r(F)$. Then we use these formulae to investigate a homomorphism between $\mathfrak{W}$ and a particular quantum group module, built by the second author in a previous paper, and show precisely when this map is well-defined for any choice of basis for $\mathfrak{W}$.
\end{abstract}



\section{Introduction}

Whittaker functions arise in $p$-adic number theory and representation theory, specifically in the study of automorphic forms over local fields and the study of principal series representations of reductive groups. They can be written in many forms: as integrals over matrix groups, as generating functions over many different combinatorial objects, as coefficients of automorphic forms, and in some cases as partition functions of lattice models. In particular, when the lattice model is solvable, this viewpoint leads to a surprising connection between the algebraic structures of the space of Whittaker functions and of modules for quantum groups.

One type of Whittaker functions of particular interest are \emph{metaplectic} Whittaker functions, which are Whittaker functions on the principal series representations of \emph{metaplectic covering groups}, central extensions of a reductive group by the $n$-th roots of unity. These groups are named after the first ``Metaplectic Group," the unique double cover of the \emph{symplectic} group $Sp_{2n}$ discovered by Weil \cite{Weil}. However, the machinery generating this particular cover can be applied in far greater generality and results in non-algebraic groups that inherit much of the interesting representation theory and number theory of their algebraic base groups. One reason for this phenomenon is that if $G$ is a group that is also a topological space, the metaplectic cover is a covering space in the topological sense as well: thus, the metaplectic covers of reductive groups, which are equipped with a topological structure, are of particular interest. These groups have been studied in various levels of generality by Kazhdan and Patterson \cite{KP}, Matsumoto \cite{Matsumoto}, Brylinski-Deligne \cite{BD}, McNamara \cite{McNrep}, Gan, Gao, and Weissman \cite{GanGao,Ast}, and many others. For our purposes, a particularly useful description is that of Brylinski-Deligne \cite{BD}, who proved that metaplectic covers of reductive $p$-adic groups are in correspondence with symmetric Weyl-group invariant bilinear forms on the cocharacter lattice. We will examine the structure of these covering groups in more detail in Section \ref{MWhit}, following the treatment of the second author in \cite{Frechette}. 

The focus of this paper is the reductive group $G = GL_r(F)$, the general linear group of $r\times r$ matrices over a nonarchimedean local field $F$ containing $\mu_{2n}$. In this case, which was first studied by Matsumoto \cite{Matsumoto}, the bilinear forms prescribed by Brylinski-Deligne \cite{BD} recover a subset of the Kahzdan-Patterson covers \cite{KP} and may be explicitly parametrized as in Frechette \cite{Frechette} as $B_{c,d}$ in terms of two parameters $c,d\in \Z$ (see Section \ref{MWhit} for the details of this construction). In general, metaplectic covers of $G$ are denoted $\widetilde{G}$, so let $\widetilde{G}_{c,d,r,n}$ be the $n$-fold cover of $GL_r(F)$ corresponding to $B_{c,d}$. It is important to note that while there may be multiple bilinear forms corresponding to a given cover, any such form will suffice for our purposes. We refer the reader to \cite{KP} or \cite{Frechette} for more detailed descriptions of which forms give identical or similar covers.


One interesting difference between the algebraic (i.e., non-metaplectic) and metaplectic cases is the dimension of the space of Whittaker functions. For a reductive algebraic group, the space of Whittaker functions on any principal series representation is one-dimensional \cite{Shalika,PS_euler,GelfandKazhdan}. In the metaplectic case, however, the construction of principal series representations becomes more complicated, due to the fact that the \emph{metaplectic torus} $\widetilde{T}$, the preimage in the metaplectic cover $\widetilde{G}$ of the torus $T(F)$, is no longer necessarily abelian. Due to this phenomenon, the dimension of the space of Whittaker functions becomes dependent on the choice of cover. As shown by McNamara \cite{McNrep}, if $\mathfrak{W}$ is the space of metaplectic Whittaker functions for a principal series representation on $\widetilde{G}$ and $H$ is the maximal abelian subgroup of $\widetilde{T},$ then
\[ \dim(\mathfrak{W}) = \left|\widetilde{T}/H\right|,\]
and the basis vectors of $\mathfrak{W}$ may be parametrized by the cosets in $\widetilde{T}/H$. Note that the space of Whittaker functions is traditionally denoted $\mathfrak{W}^{\boldsymbol{z}}$, where $\boldsymbol{z} = (z_1,...,z_r) \in \mathbb{C}^r$ lists the Satake parameters for the principal series representation. Since the results in this paper largely do not depend on the choice of $\boldsymbol{z}$, we will generally drop it from the notation and write simply $\mathfrak{W}$.

Examining the group structures of $\widetilde{T}$ and $H$ for a non-archimedean local field $F$, we achieve an explicit expression for the dimension, which we will prove in Section \ref{MWhit} as Theorem \ref{thm:restate1.1}. 

\begin{thm} 
\label{thm: Frechette}
For an $n$-fold metaplectic cover $\widetilde{G}$ of $GL_r(F)$ corresponding to the bilinear form $B_{c,d}$,
\[\left|\widetilde{T}/H\right| = \frac{n^r}{\left|\left\{\boldsymbol{x}\in (\Z/n\Z)^r: B_{c,d}(\boldsymbol{x}, \boldsymbol{y}) \equiv 0 \pmod{n} \text{ for all } \boldsymbol{y}\in(\Z/n\Z)^r\right\}\right|}.\]
\end{thm}

Our main result is a closed formula for the order of the set in the denominator of Theorem \ref{thm: Frechette}. To this end, let $$\Lambda_{fin}:=\left\{\boldsymbol{x}\in (\Z/n\Z)^r: B_{c,d}(\boldsymbol{x}, \boldsymbol{y}) \equiv 0 \pmod{n} \text{ for all } \boldsymbol{y}\in(\Z/n\Z)^r\right\}.$$ Then, using linear Diophantine equations to parametrize $\Lambda_{fin}$ in two different ways, we arrive at the following result, which is proven in two parts as Theorem \ref{thm: mainthm} and Theorem \ref{thm:Main2}, respectively.

\begin{thm*}
Given an $n$-fold metaplectic cover of $GL_r(F)$ corresponding to the bilinear form $B_{c,d}$, 
\begin{align*}
|\Lambda_{fin}|
&= \done^{r-1} \gcd\left(\dtwo, \frac{dn}{\done}\right), 
\end{align*}
where $\done = \gcd(c-d,n)$ and $\dtwo = \gcd(c+(r-1)d,n)$. Alternately, we also have that
\begin{align*}
    |\Lambda_{fin}| = \frac{\done^{r-1} \dtwo}{n} \gcd\left(\frac{n}{\done}, \frac{n}{\dtwo}, r\right) \text{lcm}\left(\frac{n}{\gcd(r,n)}, \gcd\left(\dtwo, \frac{d n}{\done} \right)\right),
\end{align*}
where $b = \gcd(r,n)$.
\label{thm: Intromainthm}
\end{thm*}

The first formula arises from viewing the parametrizing Diophantine equations, which are generated by the natural basis for the cocharacter lattice for $GL_r(F)$, as a homogeneous specialization of an inhomogenous system. This viewpoint provides a more elegant formula and a more concrete description of the structure of the space of Whittaker functions. On the other hand, while the second formula is more complicated, it arises from the root structure of $GL_r(F)$, which provides a more direct path to extending this result to other reductive groups.

For $GL_r(F)$, the space of Whittaker functions is also closely tied to a particular module for a quantum group built from the Lie algebra $\mathfrak{gl}$. Despite the name, quantum groups are not groups at all, but rather quasitriangular Hopf algebras. For this paper, we consider the quantum affine universal enveloping algebra $U_q(c,d,n):= U_q(\widehat{\mathfrak{gl}}(n/\done))$, where $q$ is the cardinality of the residue field for $F$. This quantum group has a $n/\done$-dimensional evaluation module $V_+(z)$ depending on a parameter $z\in \mathbb{C}$, whose basis vectors may be indexed using the elements of $\Z/(n/\done)\Z$.

In \cite{BBBIce}, Brubaker, Bump, and Buciumas prove that for the simplest $n$-fold metaplectic cover of $GL_r(F)$ (the one where $c=1$ and $d=0$), the space of Whittaker functions is isomorphic to an $r$-fold tensor product of evaluation modules and that after a Drinfeld twist (which changes the group action but does not affect the module structure) this isomorphism matches the action of the quantum group to the action of intertwining operators on the underlying principal series representation. The key ingredient in this proof is a lattice model construction for metaplectic Whittaker functions in the case $c=1, d=0$ developed in \cite{BBCFG} by Brubaker, Bump, Chinta, Friedberg, and Gunnells.

In \cite{Frechette}, the second author proves that both of these constructions are true in far greater generality, constructing a Whittaker function lattice model and a map $\tzmap$ between $\mathfrak{W}^{\boldsymbol{z}}$ and an $r$-fold tensor product $V_+(z_1) \otimes \cdots \otimes V_+(z_r)$ of evaluation modules for \emph{any} metaplectic cover of $GL_r(F)$. (To match to the terminology used in \cite{Frechette}, set $n_Q := n/\done$.)  Moreover, passing through this map, the action of the quantum group still matches exactly the action of intertwining operators on the Whittaker functions. 

As $c,d,r,n$ vary, the cost of dealing with more complicated covers is that this map shifts between being an isomorphism, an injection, and a surjection, and the choice of representatives for $H$-cosets affects the map. The lattice model construction used in \cite{Frechette} dictates a choice of coset representatives from $\widetilde{T}/H$ giving a basis for $\mathfrak{W}$ on which $\tzmap$ is well-defined. However, the lattice model is not necessary for the connection between $\mathfrak{W}$ and the quantum module outside of this phenomenon. 
A natural question then arises: when is the map $\tzmap$ well-defined for any choice of basis for $\mathfrak{W}$?

One of the main applications of our results is an answer for this question, using the characterization of elements of $\Lambda_{fin}$ from our proof of Main Theorem \ref{thm: Intromainthm}. Taking any basis for $\mathfrak{W}$, use Theorem \ref{thm: Frechette} to express it as a set of vectors in $(\Z/n\Z)^r$. Then the map given in \cite{Frechette} is precisely 
\begin{align*}
\theta_{\boldsymbol{z}}: \mathfrak{W}^{\boldsymbol{z}} &\rightarrow V_+(z_1) \otimes \cdots \otimes V_+(z_r)\\
\boldsymbol{\nu}\hspace{0.25cm} &\mapsto \hspace{0.5cm} \rho - \boldsymbol{\nu} \pmod{n/\done},
\end{align*}
where $\rho = (r-1,...,2,1,0)$ and the modulus is applied independently in each component of the vector. Using our characterization to show how any particular coset $\nu H$ sits within $(\Z/n\Z)^r$, we arrive at the following result, which will be proven as Theorem \ref{quantummap} and Corollary \ref{quantumIsom}.

\begin{thm*}
For a vector $\boldsymbol{z} = (z_1,...,z_r) \in \mathbb{C}^r$, the homomorphism $\tzmap$ given in \cite{Frechette}  is well-defined for any choice of basis for $\mathfrak{W}$ if and only if
\[ \gcd\left(\dtwo, \frac{dn}{\done}\right) = \gcd(c,d,n).\]
Furthermore, if $\mathfrak{W}$ is either of minimum or maximum dimension, $\tzmap$ is an isomorphism.
\end{thm*}

Understanding how this map is affected by the choice of cover is an important step to understanding how we may extend these quantum connections to metaplectic covers over other reductive groups. While the lattice model connection only exists in full for $GL_r(F)$ and $SL_r(F)$, the Whittaker function framework exists for any reductive group, so we hope that further investigation of the structure of $\mathfrak{W}$ will not only allow us to develop analogues to Main Theorem 1 for other groups, but also to determine the precise quantum group and module connected to the metaplectic Whittaker functions for other types.

Regarding the structure of this paper, in Section \ref{MWhit}, we examine the construction of metaplectic covers of $GL_r(F)$ and their Whittaker functions, culminating in a proof of Theorem \ref{thm: Frechette}. Section \ref{Cochar} introduces the first set of Diophantine equations used to parametrize $\Lambda_{fin}$, which we then use in Section \ref{SolvingCochar} to prove the first part of Main Theorem 1 as Theorem \ref{thm: mainthm}. In Section \ref{Coroot}, we introduce the second set of Diophantine equations for $\Lambda_{fin}$, which facilitate the proof of the second part of Main Theorem 1 as Theorem \ref{thm:Main2} in Section \ref{Comparisons}. In Section \ref{Structure}, we examine some cases in which the formulae for $\dim(\mathfrak{W})$ simplify dramatically and prove conditions for certain dimensions of interest for $\mathfrak{W}$, including conditions for maximum and minimum dimension. Lastly, in Section \ref{Quantum}, we develop the quantum connection and use the structure of $\mathfrak{W}$ to prove Main Theorem 2 as Theorem \ref{quantummap} and Corollary \ref{quantumIsom}.

\section*{Acknowledgements}
This project was partially supported by NSF RTG grant DMS-1745638 and was supervised by the second author as part of the University of Minnesota School of Mathematics Summer 2022 REU program. The second author is also supported by NSF grant DMS-2203042. The authors would like to thank their TA Carolyn Stephen for their guidance throughout the project, as well as Ben Brubaker and Darij Grinberg for helpful comments. 




\section{Spaces of Metaplectic Whittaker Functions}\label{MWhit}

To understand the structure of the space of metaplectic Whittaker functions, we must first concretely describe the metaplectic covers of $GL_r(F)$. We can then extend this explicit parametrization of all covers into a description of the metaplectic torus and its maximal abelian subgroup. As mentioned in the introduction, the quotient of these subgroups controls the dimension of the space of Whittaker functions: describing its structure precisely in terms of the cover allows us to reduce a complicated representation theory question to a straightforward linear algebra problem.

Suppose $n$ is a natural number and $F$ is a nonarchimedean local field containing $2n$ distinct $2n$-th roots of unity $\mu_{2n}$. Let $\mathfrak{o}$ be the ring of integers of $F$ and $\varpi$ its uniformizing element. 

\begin{definition}
Given a split reductive group $G$, an \emph{$n$-fold metaplectic cover} or \emph{$n$-fold metaplectic covering group} $\widetilde{G}$ is a non-algebraic central extension of $G$ by the $n$-th roots of unity $\mu_n$. That is, $\widetilde{G}$ is defined by the following short exact sequence:
\[ 1 \rightarrow \mu_n \rightarrow \widetilde{G} \xrightarrow{p} G \rightarrow 1.\]
\end{definition}

As a set, $\widetilde{G}$ is the set of tuples $(\zeta, g)$ where $\zeta \in \mu_n, g\in G$. However, group multiplication is controlled by a cocycle $\sigma \in H^2(G,\mu_n)$; that is, for two elements $(\zeta_1,g_1), (\zeta_2,g_2)$, their product in $\widetilde{G}$ is
\[(\zeta_1,g_1)\cdot (\zeta_2,g_2) = (\zeta_1\zeta_2\sigma(g_1,g_2), g_1g_2).
\]

In the process of writing down an explicit form for cocycles for covers of $GL_r(F)$, we see that a slightly more general case may be handled simultaneously. Set $G=GL_r(F)$ for the remainder of the paper.

\begin{definition}
More generally, a \emph{metaplectic covering group essentially of degree $n$} is given by a short exact sequence
\[ 1 \rightarrow \mu_m \rightarrow \widetilde{G} \xrightarrow{p} G \rightarrow 1\]
where $n|m$ and the corresponding cocycle $\sigma \in H^2(G,\mu_m)$ satisfies the property that $[\sigma^n]$ is trivial in $H^2(G,\mathbb{C}^\times)$ under the inclusion induced by an embedding $\varepsilon: \mu_m \rightarrow \mathbb{C}^\times$.
\end{definition}

While it is slightly tedious to write down formulae for these cocycles on general elements of $G$, their expressions over the torus $T$ of diagonal matrices in $G$ are quite elegant. In \cite{Frechette}, the second author proves that all metaplectic covers essentially of degree $n$ over $GL_r(F)$ come from a cocycle of the form
\begin{equation} \label{eqn:cocycle} \sigma_{c,d}(\boldsymbol{x},\boldsymbol{y}) = \left(\det(\boldsymbol{x}),\det(\boldsymbol{x})\right)_{2n}^c \prod_{i>j}\left(x_i,y_j\right)^{d-c}_n.
\end{equation}
for $c,d\in \Z$, where $\boldsymbol{x}, \boldsymbol{y} \in T$ and $(\cdot,\cdot)_{k}$ denotes the $k$-th Hilbert symbol (see Neukirch \cite{Neukirch} for more details on the construction of Hilbert symbols). Notably, making the shift to covers essentially of degree $n$ rather than ``purely" of degree $n$ is necessary to include the metaplectic cover corresponding to the cocycle $\sigma_{1,0}$, which, while only essentially of degree $n$, has been an integral example for this field (see for example \cite{BBBIce,BBCFG,MetIce,BBFbook, McNcrystal}.) 

\begin{remark}
Since the $2n$-th Hilbert symbol produces $2n$-th roots of unity, it is necessary that $F$ contain $\mu_{2n}$ for the group to be well defined. However, if we are considering a cocycle for which the parameter $c$ is even, we may relax this condition and require $F$ to contain only $\mu_n$.
\end{remark}

In \cite{BD}, Brylinski-Deligne prove that the set of metaplectic covers is in correspondence with the set of symmetric Weyl-group invariant bilinear forms $B:Y\times Y \rightarrow \Z$ on the cocharacter lattice $Y$ such that $\frac{B(\alpha^\vee,\alpha^\vee)}{2} \in \Z$ for all coroots $\alpha^\vee\in Y$. For $G=GL_r(F)$, a natural choice of basis for $Y$ is the set of $r$ fundamental coweights $\varepsilon_i^\vee: F^\times \rightarrow T$, for $i=1,...,r$, in which $\varepsilon_i^\vee(a) := $ diag$(1,...,1,a,1,...,1)$, where $a$ is in the $i$-th entry. Note: while we will use the notation $\lambda(a)$ for $\lambda\in Y, a\in F^\times$, another common notation is $a^\lambda$.

Under this basis, the cocharacter lattice $Y$ is isomorphic to $\Z^r$; for instance,
\[ (\varepsilon_1^\vee+ 3\varepsilon_2^\vee)(a) = \text{diag}(a, a^3,1,...,1).\]

Using this basis, we represent a bilinear form on $Y$ in terms of the corresponding matrix $A$ such that for $\boldsymbol{x},\boldsymbol{y}\in Y$,
\[ B(\boldsymbol{x},\boldsymbol{y}) = \boldsymbol{x}^T A \boldsymbol{y}.\]
Each of the conditions from the Brylinski-Deligne correspondence translates into a condition for this matrix. First, a symmetric bilinear form prescribes a symmetric matrix. Second, the Weyl group $W$ is isomorphic to the symmetric group $S_r$, and acts on $Y$ by $\sigma \cdot \varepsilon_i^\vee = \varepsilon_{\sigma(i)}^\vee$. Thus, $A$ must be invariant under conjugation by permutation matrices, so for some (suggestively named) $c,d\in \Z$, we have $a_{i,i} = c$ for all $i$ and $a_{i,j} = d$ for all $i\neq j$. (See the matrix in (\ref{eqn:Bcdrn}) for an illustration of this requirement.)

There are $r-1$ simple coroots, each of the form $\varepsilon_i^\vee - \varepsilon_{i+1}^\vee$. To check the integrality condition on the coroot lattice, it suffices to show that it holds for simple coroots. However, for $GL_r(F)$, this condition is satisfied already: for any simple coroot $\varepsilon_i^\vee,$
\[ \frac{B_{c,d}(\varepsilon_i^\vee, \varepsilon_i^\vee)}{2} = \frac{2(c-d)}{2} = c-d \in \Z.\]

By Brylinski-Deligne, the metaplectic cover corresponding to this bilinear form satisfies the following condition: if $\boldsymbol{x}, \boldsymbol{y} \in \widetilde{T} = p^{-1}(T)$ such that $p(\boldsymbol{x}) = \lambda(x), p(\boldsymbol{y}) = \mu(y)$ for some $x,y\in F^\times$ and  $\lambda,\mu \in Y$, then the commutator of $\boldsymbol{x}$ and $\boldsymbol{y}$ is
\[[\boldsymbol{x},\boldsymbol{y}] = (x,y)_n^{B(\lambda,\mu)}.\]

Evaluating the commutator in terms of an explicit cocycle, we may identify the bilinear form corresponding to a specific cocycle and vice versa. Note that this property illuminates one of the key reasons the Brylinski-Deligne correspondence is not a bijection: since the cocycle in (\ref{eqn:cocycle}) depends on powers of Hilbert symbols, there are many different cocycles which will give exactly the same cover, specifically any $\sigma_{c',d'}$ such that $c'\equiv c \pmod{2n}$ and $d'-c' \equiv d-c\pmod{n}.$

\begin{thm}[Frechette \cite{Frechette}] \label{thrm: background}
For $c,d\in \Z$, the essentially $n$-fold metaplectic cover of $GL_r(F)$ with multiplication given by $\sigma_{c,d}$ corresponds to the bilinear form $B_{c,d}$ that acts on $(\boldsymbol{x},\boldsymbol{y})\in \Z^r \times \Z^r$ by
\begin{equation}\label{eqn:Bcdrn}
B_{c,d}(\boldsymbol{x},\boldsymbol{y}) = \boldsymbol{x}^T \cdot \begin{pmatrix}
    c & d & d & \dots & d\\
    d & c & d & \dots & d\\
    d & d & c & \dots & d\\
    \vdots & \vdots & & \ddots & \vdots\\
    d & d & d & \dots & c
    \end{pmatrix}\cdot \boldsymbol{y}.
\end{equation} 
\end{thm}

Conflating the bilinear form with its corresponding matrix, we will denote both by $B_{c,d}$; we hope this abuse of notation will not cause any confusion. Note: in \cite{Frechette}, this bilinear form is parametrized slightly differently as $B_{b,c}$, where $b = c-d$.

Now that we have an explicit description of our metaplectic covers, we investigate what this parametrization tells us about space of metaplectic Whittaker functions. For the purposes of this paper, we will not need the constructions of the metaplectic Whittaker functions themselves, nor those of the metaplectic principal series representations on which they are defined. Instead, we will use the following theorem of McNamara to investigate the space of Whittaker functions through its connection to the metaplectic torus. For the definitions of the metaplectic principal series representations and their Whittaker functions, we refer the reader to Sections 6 and 8, respectively, of \cite{McNrep} as a convenient source.

\begin{thm}[McNamara \cite{McNrep}]
Fix a metaplectic cover $\widetilde{G}$ over a $p$-adic reductive group $G$ and let $\mathfrak{W}$ be the space of metaplectic Whittaker functions for a principal series representation on $\widetilde{G}$. Let the \emph{metapletic torus} $\widetilde{T}$ be the preimage in $\widetilde{G}$ of the torus $T(F)$, and let $H$ be the maximal abelian subgroup of $\widetilde{T}$. Then
\[ \dim(\mathfrak{W}) = \left|\widetilde{T}/H\right|.\]
\end{thm} 
Note: the group $T$ of diagonal matrices is denoted $T$ because it is an abelian \emph{torus}, that is, it is isomorphic to $(F^\times)^r$. While we call $\widetilde{T}$ the \emph{metaplectic torus}, it is no longer abelian, nor is it technically a torus, as its elements are $(\zeta, t)$ where $\zeta \in \mu_n$ (where $\mu_n \subsetneq F$) and $t\in T$. Investigating the precise structure of $\widetilde{T}$, we prove the following theorem, which is a restatement of Theorem \ref{thm: Frechette}.

\begin{thm} \label{thm:restate1.1}
For a metaplectic cover $\widetilde{G}$ of $GL_r(F)$ corresponding to the bilinear form $B_{c,d}$,
\[\left|\widetilde{T}/H\right| = \frac{n^r}{\left|\left\{\boldsymbol{x}\in (\Z/n\Z)^r: B_{c,d}(\boldsymbol{x}, \boldsymbol{y}) \equiv 0 \pmod{n} \text{ for all } \boldsymbol{y}\in(\Z/n\Z)^r\right\}\right|}.\]
\end{thm}

\begin{proof}
Using our description of the metaplectic covers, we can express the subgroups $\widetilde{T}$ and $H$ more explicitly: using the Iwasawa decomposition of $GL_r(F)$, we have that $\widetilde{T} = \mu_n \times T(\mathfrak{o}) \times Y$ as a set. That is, for $(\zeta, t) \in T,$ we may write $t = t_0 \cdot \lambda(\varpi)$ for some $t_0\in T(\mathfrak{o})$ and $\lambda \in Y$. 

Since $H$ is a subgroup of $\widetilde{T}$, its elements also look like $(\zeta, h)$ where $\zeta$ is an $n$-th root of unity and $h$ is a diagonal matrix with entries in $F$. Examining the group law on $\widetilde{G}$, we see that the root of unity does not impede commutativity of elements, so it is the matrix component $h$ we must examine further to obtain a description of $H$. To do so, recall that $\mathfrak{o}$ is the valuation ring of $F$ and $\varpi$ the uniformizing element. Then by \cite{McNrep}, as a set we have $H = \mu_n\times T(\mathfrak{o})\times \Lambda$, where  $\Lambda$ is the free abelian group 
\[ \Lambda := \{\lambda\in Y: s(\lambda(\varpi)) \in H\}\]
for $s: G\rightarrow \widetilde{G}$ the standard section $s(g) = (1,g)$.
Using the commutator relation and the fundamental coweight basis for $Y$, an equivalent description for $\Lambda$ is
\begin{equation}\label{Lambdadef}
 \Lambda = \left\{\boldsymbol{x}\in \Z^r: B_{c,d}(\boldsymbol{x}, \boldsymbol{y}) \equiv 0 \pmod{n} \text{ for all } \boldsymbol{y}\in\Z^r\right\}.
 \end{equation}

It is useful to think of the group $\Lambda$ as controlling the powers of $\varpi$ in each entry on the diagonal of the matrix $h$. That is, for any element $(\zeta, h) \in H$, we have $h = h_0 \cdot \text{diag}(\varpi^{\lambda_1},...,\varpi^{\lambda_r})$ where $h_0 \in T(\mathfrak{o})$ and $\lambda = \lambda_1\varepsilon_1^\vee + \cdots+ \lambda_r\varepsilon_r^\vee$ is in $\Lambda$. 

Then, combining our descriptions of $\widetilde{T}$ and $H$ to consider $\widetilde{T}/H$, we see that 
\[ |\widetilde{T}/H| = \left|Y/\Lambda\right| = \left| \Z^r/ \Lambda\right|, \]
where the last description uses the embedding of $\Lambda$ in $\Z^r$ described in \eqref{Lambdadef}.
Notice that if $\lambda_i \in n\Z$ for all $i$, then $B((\lambda_1,...,\lambda_r),\boldsymbol{y})$ will automatically be a multiple of $n$ for any $\boldsymbol{y}\in \Z^r$, and therefore $\lambda = \lambda_1\varepsilon_1^\vee + \cdots+ \lambda_r\varepsilon_r^\vee$ will be in $\Lambda$. Therefore, it suffices to consider all coordinates $\lambda_i$ mod $n$, and so
\[|\widetilde{T}/H| = \left|(\Z/n\Z)^r / \left(\Lambda \cap (\Z/n\Z)^r\right) \right|\]
which completes the proof.\qedhere
\end{proof}

Let $\Lambda_{fin}:=\left\{\boldsymbol{x}\in (\Z/n\Z)^r: B_{c,d}(\boldsymbol{x}, \boldsymbol{y}) \equiv 0 \pmod{n} \text{ for all } \boldsymbol{y}\in(\Z/n\Z)^r\right\}.$ We will spend the next several sections developing two related systems of linear Diophantine equations which allow us to describe the elements in $\Lambda_{fin}$, each of which will give us a distinct formula for $|\Lambda_{fin}|$. We will then return to the broader framework in Section \ref{Structure} to what these different formulae tell us about the structure of $|\widetilde{T}/H|$ and thus the structure of $\mathfrak{W}$.

\section{Cocharacter Diophantine Equations and Phenomena}\label{Cochar}

In this section, we use the natural basis for the cocharacter lattice $Y$ of $GL_r(F)$ to develop a set of $r$ linear Diophantine equations in terms of $c,d,$ and $n$ that describe the set $\Lambda_{fin}$, which we call the \emph{cocharacter equations}. This perspective turns a representation theoretic question into a linear algebra one, where altering each of the parameters $c,d,r,$ and $n$ has a different effect on the system. We also take time now to develop a visual framework which illuminates this distinction in the roles of each of our parameters.

Examining the conditions for $\Lambda_{fin}$ using the viewpoint of the fundamental coweight basis for $Y$ (see Section \ref{MWhit}), we arrive at the following system of $r$ equations. Let $\boldsymbol{0}_r = (0, 0, \dots, 0)^T$ be the $r \times 1$ column vector with all entries equal to $0$, and define $\boldsymbol{1}_r = (1, 1, \dots, 1)^T$ similarly. Recall that $B_{c,d}$ is both the bilinear form given in Theorem \ref{thrm: background} and its corresponding $r \times r$ matrix.

\begin{definition}
For natural numbers $r, n \geq 1$ and constants $c, d \in \Z$, we call the following system of equations the \textbf{cocharacter equations}:
\begin{equation*}
    B_{c,d} \cdot \boldsymbol{x}
    = \boldsymbol{0}_r  \pmod{n}.
\end{equation*}
That is, for $\boldsymbol{x} = (x_1,...,x_r)^T$, we have
\begin{align*}
    cx_1 + dx_2 + \cdots + dx_r &\equiv 0\pmod{n}\\
    dx_1 + cx_2 + \cdots + dx_r &\equiv 0\pmod{n}\\
    &\vdots\\
    dx_1 + dx_2 + \cdots + cx_r &\equiv 0\pmod{n}
\end{align*}
\end{definition}

Here, the $i$-th equation arises from evaluating $\boldsymbol{x}\in Y$ against $\varepsilon_i^\vee$ in the bilinear form $B_{c,d}$ for each $i\in\{1,...,r\}$. Thus, Lemma \ref{lemma: lambdasize} follows directly.

\begin{lemma}
\label{lemma: lambdasize}
Let $S_{cochar}(c,d,r,n)$ be the number of solutions $\boldsymbol{x}\in (\Z/n\Z)^r$ to the cocharacter equations. Then, for the cover $\widetilde{G}_{c,d,r,n}$, we have $S_{cochar}(c,d,r,n) = |\Lambda_{fin}|$.
\end{lemma}

Looking at the values of $S_{cochar}$ for a fixed $r$ and $n$ as we range over $c$ and $d$, certain patterns emerge which motivate defining constants which we call the \emph{diagonal numbers}. These constants will be fundamental in our formulas for $S_{cochar}$, so we take the time to explore them now.

For a fixed $r$ and $n$, note first that it suffices to understand $S_{cochar}$ for $c,d\pmod{n}$, as $S_{cochar}(c,d,r,n) = S_{cochar}(c',d,r,n)$ for $c\equiv c'\pmod{n}$ and likewise for $d$. It will be useful to visualize the values of $S_{cochar}$ as a table ranging over $c,d \in \Z/n\Z$ in the following manner:

\begin{center}
\small\begin{tikzpicture}

\draw[] (-.75,4.25) -- (3.5,4.25);
\draw[] (0,1.5) -- (0,5);
\draw[black!50] (-.5,3.75) -- (3.5,3.75);
\draw[black!50] (-.5,3.25) -- (3.5,3.25);
\draw[black!50] (-.5,2.75) -- (3.5,2.75);
\draw[black!50] (-.5,2.1) -- (3.5,2.1);
\draw[] (-.5,1.5) -- (3.5,1.5);

\draw[black!50] (0.5,1.5) -- (0.5,4.75);
\draw[black!50] (1,1.5) -- (1,4.75);
\draw[black!50] (1.5,1.5) -- (1.5,4.75);
\draw[black!50] (2.125,1.5) -- (2.125,4.75);
\draw[] (3.5,1.5) -- (3.5,4.75);

\node at (-.25,4.875) {$d$};
\node at (0.25,4.5) {0};
\node at (0.75,4.5) {1};
\node at (1.25,4.5) {2};
\node at (1.875,4.5) {$\cdots$};
\node at (2.75,4.5) {$(n-1)$};

\node at (-0.675,4.5) {$c$};
\node at (-.25,4) {0};
\node at (-.25,3.5) {1};
\node at (-.25,3) {2};
\node at (-.25,2.5) {$\vdots$};
\node at (-.6,1.75) {$(n-1)$};

\end{tikzpicture}
\end{center}

\begin{figure}[h]
    \centering
\scalebox{0.9}{
\begin{tikzpicture}

\node at (-1,0) {$\begin{matrix}
49 & 1 & 1 & 1 & 1 & 1 & 1\\
1 & 7 & 1 & 1 & 1 & 1 & 7\\
1 & 1 & 7 & 1 & 1 & 7 & 1\\
1 & 1 & 1 & 7 & 7 & 1 & 1\\
1 & 1 & 1 & 7 & 7 & 1 & 1\\
1 & 1 & 7 & 1 & 1 & 7 & 1\\
1 & 7 & 1 & 1 & 1 & 1 & 7
\end{matrix}$};
\draw[red] (-2.2,-1.45) -- (0.95,1.2);
\draw[red] (-2.8,1.5) -- (0.95, -1.4);
\draw[red] (-2.8,1) -- (-2.3, 1.5);
\node at (-1,-2.2){$r=2,$ $n=7$};

\node at (4.5,0) {$\begin{matrix}
64 & 1 & 4 & 1 & 16 & 1 & 4 & 1\\
1 & 8 & 1 & 8 & 1 & 8 & 1 & 8\\
4 & 1 & 16 & 1 & 4 & 1 & 16 & 1\\
1 & 8 & 1 & 8 & 1 & 8 & 1 & 8\\
16 & 1 & 4 & 1 & 32 & 1 & 4 & 1\\
1 & 8 & 1 & 8 & 1 & 8 & 1 & 8\\
4 & 1 & 16 & 1 & 4 & 1 & 16 & 1\\
1 & 8 & 1 & 8 & 1 & 8 & 1 & 8
\end{matrix}$};
\draw[red] (2.7,-1.6) -- (6.9,1.2);
\draw[Green] (2.1,-1.2) -- (6.3,1.6);
\draw[blue] (2.1,-0.4) -- (5,1.6);
\draw[Green] (2.1,0.4) -- (3.8,1.6);

\draw[Green] (4.1,-1.6) -- (6.9,0.4);
\draw[blue] (5.3,-1.6) -- (6.9,-0.4);

\draw[red] (2.1,1.65) -- (6.9, -1.6);
\draw[red] (2.1,1.3) -- (2.6, 1.65);
\draw[Green] (2.1,0.8) -- (5.6, -1.6);
\draw[blue] (2.1,0) -- (4.5, -1.6);
\draw[Green] (2.1,-0.8) -- (3.2, -1.6);

\draw[Green] (3.5,1.65) -- (6.9, -0.8);
\draw[blue] (4.6,1.65) -- (6.9, 0.1);
\draw[Green] (6,1.65) -- (6.9, 0.9);
\node at (4.5,-2.2){$r=2$, $n=8$};

\node at (11.5,0) {$\begin{matrix}
81 & 1 & 1 & 9 & 1 & 1 & 9 & 1 & 1\\
1 & 9 & 3 & 1 & 3 & 3 & 1 & 3 & 9\\
1 & 3 & 9 & 1 & 3 & 3 & 1 & 9 & 3\\
9 & 1 & 1 & 27 & 1 & 1 & 27 & 1 & 1\\
1 & 3 & 3 & 1 & 9 & 9 & 1 & 3 & 3\\
1 & 3 & 3 & 1 & 9 & 9 & 1 & 3 & 3\\
9 & 1 & 1 & 27 & 1 & 1 & 27 & 1 & 1\\
1 & 3 & 9 & 1 & 3 & 3 & 1 & 9 & 3\\
1 & 9 & 3 & 1 & 3 & 3 & 1 & 3 & 9
\end{matrix}$};
\draw[red] (9.6,-1.8) -- (14.1,1.5);
\draw[blue] (9,0.3) -- (11.2,1.9);
\draw[blue] (9,-1) -- (13,1.9);

\draw[blue] (11.4,-1.8) -- (14.1,0.2);
\draw[blue] (13.2,-1.8) -- (14.1,-1.1);

\draw[red] (9,1.9) -- (14.1, -1.8);
\draw[red] (9,1.5) -- (9.4, 1.9);
\draw[blue] (9,0.6) -- (12.3, -1.8);
\draw[blue] (9,-0.6) -- (10.5, -1.8);

\draw[blue] (10.8,1.9) -- (14.1, -0.6);
\draw[blue] (12.5,1.9) -- (14.1, 0.6);

\node at (11.5,-2.2){$r=2$, $n=9$};

\node at (1,-5.5) {$\begin{matrix}
729 & 1 & 1 & 27 & 1 & 1 & 27 & 1 & 1\\
1 & 81 & 1 & 1 & 27 & 1 & 1 & 27 & 1\\
1 & 1 & 81 & 1 & 1 & 27 & 1 & 1 & 27\\
27 & 1 & 1 & 243 & 1 & 1 & 27 & 1 & 1\\
1 & 27 & 1 & 1 & 81 & 1 & 1 & 27 & 1\\
1 & 1 & 27 & 1 & 1 & 81 & 1 & 1 & 27\\
27 & 1 & 1 & 27 & 1 & 1 & 243 & 1 & 1\\
1 & 27 & 1 & 1 & 27 & 1 & 1 & 81 & 1\\
1 & 1 & 27 & 1 & 1 & 27 & 1 & 1 & 81\\
\end{matrix}$};
\node at (1,-8){$r=3$, $n=9$};
\draw[red] (-2.4,-3.6) -- (4.4,-7.3);
\draw[blue] (-0.1,-3.6) -- (4.4,-6.1);
\draw[blue] (2.2,-3.6) -- (4.4,-4.8);
\draw[blue] (-2.3,-4.8) -- (2,-7.3);
\draw[blue] (-2.3,-6.2) -- (-0.2,-7.3);

\draw[red] (-2.3,-4.1) -- (-1.7, -3.5);
\draw[red] (-1.8,-7.4) -- (1.6, -3.5);
\draw[red] (1.5,-7.4) -- (4.4,-4.3);
\draw[blue] (-2.3,-5.4) -- (-0.8, -3.6);
\draw[blue] (-2.3,-6.6) -- (0.5, -3.6);

\draw[blue] (-0.8,-7.4) -- (2.8, -3.6);
\draw[blue] (0.5,-7.4) -- (3.9, -3.6);
\draw[blue] (2.8,-7.4) -- (4.4, -5.5);

\node at (10,-5.5) {$\begin{matrix}
4096 & 1 & 16 & 1 & 256 & 1 & 16 & 1\\
1 & 512 & 1 & 16 & 1 & 128 & 1 & 16\\
16 & 1 & 1024 & 1 & 16 & 1 & 256 & 1\\
1 & 16 & 1 & 512 & 1 & 16 & 1 & 128\\
256 & 1 & 16 & 1 & 2048 & 1 & 16 & 1\\
1 & 128 & 1 & 16 & 1 & 512 & 1 & 16\\
16 & 1 & 256 & 1 & 16 & 1 & 1024 & 1\\
1 & 16 & 1 & 128 & 1 & 16 & 1 & 512\\
\end{matrix}$};
\draw[red] (6.1,-3.8) -- (13.9,-7.1);
\draw[Green] (6.3,-4.6) -- (12,-7.1);
\draw[blue] (6.3,-5.5) -- (10,-7.1);
\draw[Green] (6.3,-6.3) -- (8,-7.1);
\draw[red] (6.2,-4.5) -- (6.8,-3.7);
\draw[Green] (8.2,-3.8) -- (13.7,-6.3);
\draw[blue] (10,-3.8) -- (13.7,-5.4);
\draw[Green] (12.1,-3.8) -- (13.7,-4.5);

\draw[red] (6.9,-7.1) -- (9.3,-3.8);
\draw[red] (9.5,-7.2) -- (11.9,-3.8);
\draw[red] (12.1,-7.1) -- (13.8,-4.8);

\draw[Green] (6.2,-5.3) -- (7.4,-3.8);
\draw[blue] (6.2,-6.1) -- (8.1,-3.8);
\draw[Green] (6.2,-7.1) -- (8.7,-3.8);
\draw[Green] (7.5,-7.1) -- (9.9,-3.8);
\draw[blue] (8.2,-7.1) -- (10.6,-3.8);
\draw[Green] (8.9,-7.1) -- (11.3,-3.8);
\draw[Green] (10.1,-7.1) -- (12.5,-3.8);
\draw[blue] (10.8,-7.1) -- (13.2,-3.8);
\draw[Green] (11.4,-7.1) -- (13.7,-4.1);
\draw[Green] (12.8,-7.1) -- (13.7, -5.8);

\node at (10,-8){$r=4$, $n=8$};
\end{tikzpicture}}
    \caption{Examples of the Cocharacter Phenomena for different choices of $r$ and $n$. In each example, notice that there is one set of diagonals of slope -1 and one of slope $r-1$: the former indicate the effect of the \emph{first diagonal numbers} $\done$ and the latter that of the \emph{second diagonal numbers} $\dtwo$. Diagonals for the same diagonal numbers (greater than 1) are marked with the same color within each example. For instance, in the second example ($r=2,n=8$) red marks diagonal numbers equal to 8, blue equal to 4, and green equal to 2.}
    \label{fig:diagexamples}
\end{figure}
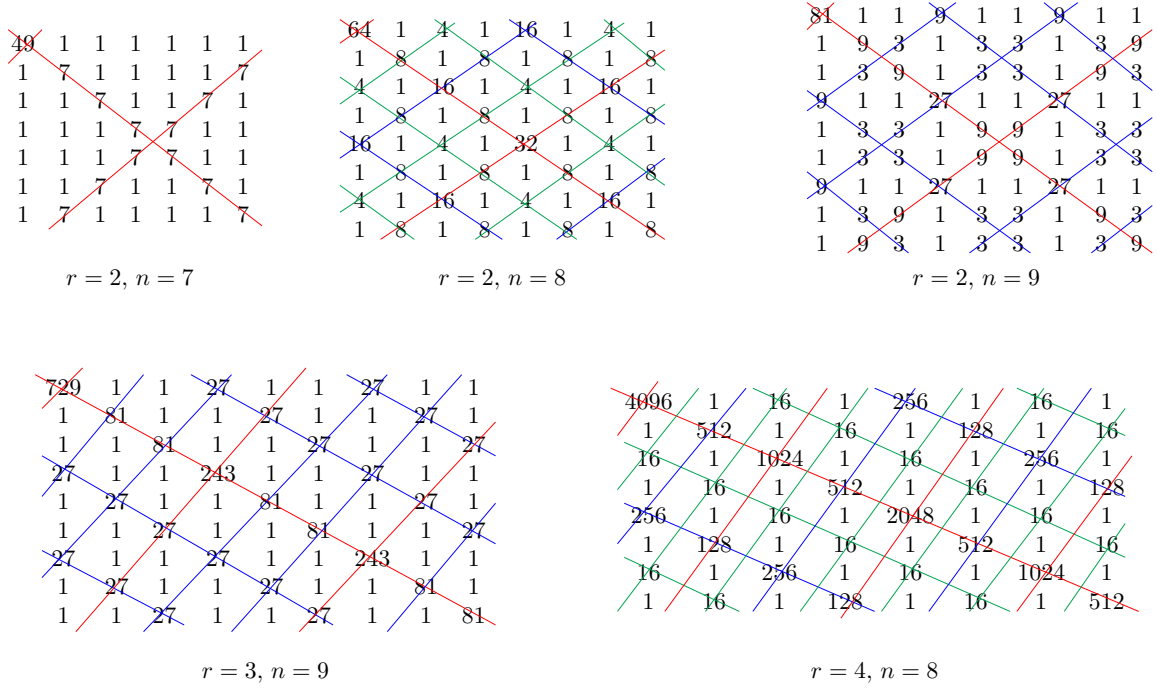

Examining Figure \ref{fig:diagexamples}, which contains several examples of these tables, notice that the values of $S_{cochar}$ on the marked diagonals in each picture are each divisible by common factors and that there are two sets of diagonals in each picture. Motivated by this phenomena, we assign each entry a set of two diagonal numbers.

\begin{definition}
Let $\done = gcd(c-d,n)$ be the \emph{first diagonal number} and define $\dtwo = gcd(c+(r-1)d,n)$ to be the \emph{second diagonal number}.
\end{definition}

Note that for a specific entry in place $c,d$, its first diagonal number captures the column $c-d$ where its diagonal of slope $-1$ intersects the first row and similarly, the second diagonal number identifies the row where its diagonal of slope $r-1$ intersects the first column.

\begin{example}
When $r$ and $n$ are coprime, the table for $S_{cochar}$ depends solely on these diagonal numbers, which we will later prove in Section \ref{Comparisons} (see Corollary \ref{cor:coprime}.) For instance, the table where $n = 10,$ $r = 3$ is shown in Figure \ref{diag-figure}, with diagonal numbers marked, and the value of every entry in this matrix is determined by its two diagonal numbers. Specifically, we have $S_{cochar}(10,3,c,d) = \done^{r-1}\dtwo = \done^2\dtwo$ for any $c,d$.

\begin{figure}[h]
    \centering
\begin{tikzpicture}

\small \matrix (m) [matrix of nodes,ampersand replacement=\&]
{
1000 \& 2 \& 8 \& 2 \& 8 \& 250 \& 8 \& 2 \& 8 \& 2\\
1 \& 100 \& 5 \& 4 \& 1 \& 4 \& 25 \& 20 \& 1 \& 4 \\
8 \& 2 \& 200 \& 2 \& 40 \& 2 \& 8 \& 50 \& 8 \& 10 \\
1 \& 20 \& 1 \& 100 \& 1 \& 4 \& 5 \& 4 \& 25 \& 4 \\
8 \& 2 \& 8 \& 10 \& 200 \& 2 \& 8 \& 2 \& 40 \& 50 \\
125 \& 4 \& 1 \& 4 \& 1 \& 500 \& 1 \& 4 \& 1 \& 4 \\
8 \& 50 \& 40 \& 2 \& 8 \& 2 \& 200 \& 10 \& 8 \& 2 \\
1 \& 4 \& 25 \& 4 \& 5 \& 4 \& 1 \& 100 \& 1 \& 20 \\
8 \& 10 \& 8 \& 50 \& 8 \& 2 \& 40 \& 2 \& 200 \& 2 \\
1 \& 4 \& 1 \& 20 \& 25 \& 4 \& 1 \& 4 \& 5 \& 100 \\
};

\draw[red]  (m-1-1.north west) -- (m-10-10.south east);
\draw[blue] (m-6-1.north west) -- (m-10-5.south east);
\draw[blue] (m-1-6.north west) -- (m-5-10.south east);
\draw[LimeGreen] (m-3-1.north west) -- (m-10-8.south east);
\draw[LimeGreen] (m-5-1.north west) -- (m-10-6.south east);
\draw[LimeGreen] (m-7-1.north west) -- (m-10-4.south east);
\draw[LimeGreen] (m-9-1.north west) -- (m-10-2.south east);
\draw[LimeGreen] (m-1-3.north west) -- (m-8-10.south east);
\draw[LimeGreen] (m-1-5.north west) -- (m-6-10.south east);
\draw[LimeGreen] (m-1-7.north west) -- (m-4-10.south east);
\draw[LimeGreen] (m-1-9.north west) -- (m-2-10.south east);

\draw[red] (m-10-1.south east) -- (m-1-6.north east);
\draw[red] (m-10-6.south east) -- (m-3-10.north east);
\draw[blue] (m-10-4.south west) -- (m-1-9.north west);
\draw[blue] (m-6-1.south west) -- (m-1-4.north west);
\draw[blue] (m-10-9.south west) -- (m-8-10.north east);
\draw[red] (m-1-1.south west) -- (m-1-1.north);
\draw[LimeGreen] (m-3-1.south west) -- (m-1-2.north east);
\draw[LimeGreen] (m-5-1.south west) -- (m-1-3.north east);
\draw[LimeGreen] (m-7-1.south west) -- (m-1-4.north east);
\draw[LimeGreen] (m-9-1.south west) -- (m-1-5.north east);
\draw[LimeGreen] (m-10-2.south east) -- (m-1-7.north east);
\draw[LimeGreen] (m-10-3.south east) -- (m-1-8.north east);
\draw[LimeGreen] (m-10-4.south east) -- (m-1-9.north east);
\draw[LimeGreen] (m-10-5.south east) -- (m-1-10.north east);
\draw[LimeGreen] (m-10-7.south east) -- (m-5-10.north east);
\draw[LimeGreen] (m-10-7.south east) -- (m-5-10.north east);
\draw[LimeGreen] (m-10-8.south east) -- (m-7-10.north east);
\draw[LimeGreen] (m-10-9.south east) -- (m-9-10.north east);

\node [xshift=-4.25cm, yshift=2.5cm](m-1-1){$\done=$};
\node [red, xshift=-3.75cm, yshift=2.5cm](m-1-1){$10$};
\node [xshift=-2.9cm, yshift=2.5cm](m-1-1){$1$};
\node [LimeGreen, xshift=-2.1cm, yshift=2.5cm](m-1-1){$2$};
\node [xshift=-1.25cm, yshift=2.5cm](m-1-1){$1$};
\node [LimeGreen, xshift=-0.65cm, yshift=2.5cm](m-1-1){$2$};
\node [blue, xshift=-0cm, yshift=2.5cm](m-1-1){$5$};
\node [LimeGreen, xshift=.75cm, yshift=2.5cm](m-1-1){$2$};
\node [xshift=1.5cm, yshift=2.5cm](m-1-1){$1$};
\node [LimeGreen, xshift=2.1cm, yshift=2.5cm](m-1-1){$2$};
\node [xshift=2.9cm, yshift=2.5cm](m-1-1){$1$};
\node [xshift=-4.7cm, yshift=1.65cm](m-1-1){$\dtwo=$};
\node [red, xshift=-4cm, yshift=1.65cm](m-1-1){$10$};
\node [xshift=-4cm, yshift=1.2cm](m-1-1){$1$};
\node [LimeGreen, xshift=-4cm, yshift=0.8cm](m-1-1){$2$};
\node [xshift=-4cm, yshift=.4cm](m-1-1){$1$};
\node [LimeGreen, xshift=-4cm, yshift=-0.1cm](m-1-1){$2$};
\node [blue, xshift=-4cm, yshift=-.5cm](m-1-1){$5$};
\node [LimeGreen, xshift=-4cm, yshift=-1cm](m-1-1){$2$};
\node [xshift=-4cm, yshift=-1.35cm](m-1-1){$1$};
\node [LimeGreen, xshift=-4cm, yshift=-1.75cm](m-1-1){$2$};
\node [xshift=-4cm, yshift=-2.1cm](m-1-1){$1$};

\draw[violet, line width=3pt] (m-5-9) circle (3mm);

\coordinate (P1) at (11.5,-1);

\node [violet, xshift=-5cm, yshift=2.1cm] at  (P1) {$\done=\textcolor{LimeGreen}{2}, \dtwo= \textcolor{red}{10}$};
\node [violet, xshift=-5cm, yshift=2.1cm] at  (11.5,-1.5) {$40 = 2^2 \cdot 10$};
\draw [violet,thick,->] (5,1) -- (3.15,0.4);
\end{tikzpicture}
  \caption{The table showing $S_{cochar}(c,d,3,10)$ for all $(c,d) \in \Z_{10} \times \Z_{10}$ with diagonals for diagonal numbers greater than 1 marked. Notice here that since $r=3$ and $n=10$ are coprime, every entry is equal to $\done^2\cdot \dtwo$. In contrast, see the example in Figure \ref{fig:diagexamples} for $r=3$ and $n=9$, where this is not true.}
  \label{diag-figure}
\end{figure}
\end{example}

In general, given a random $n$ and $r$, the value of $S_{cochar}$ will not depend nearly so simply on $\done$ and $\dtwo$, but they still play an important determining role. To find a closed formula for $S_{cochar}$, we must look to an inhomogenous generalization of the homogenous cocharacter equations with which we started.

\begin{definition}
Let $a \in \Z/n\Z$, and $\boldsymbol{x} \in (\Z/n\Z)^r$. Then the \emph{inhomogenous cocharacter equations} for $a\in \Z$ are defined by
    \begin{equation}
          B_{c,d} \cdot \boldsymbol{x}
    = a \cdot \boldsymbol{1}_r \pmod{n}.
    \label{inhom}
    \end{equation}

Let $S_{inhom}(c,d,r,n)$ be the number of total solutions to the inhomogenous cocharacter equations, ranging over all values of $a\in \Z$.
\end{definition}

In the next section, we will solve for $S_{cochar}(c,d,r,n)$ by characterizing the set of solutions to
the inhomogenous cocharacter equations using straightforward linear algebra techniques and identifying the proportion of solutions with $a \equiv 0\pmod{n}$. To do this, we will need to identify a precise formula for smallest nonzero value of $a$ for which \eqref{inhom} has a solution.

\begin{definition}
For a fixed $c,d,r,n$, let $A(c,d,r,n)$ be the smallest positive integer value for $a$ such that there is a solution to the inhomogenous cocharacter equations (\ref{inhom}). 
\end{definition}







\section{Proof of Main Theorem 1 Part 1}\label{SolvingCochar} 

For the entirety of this section, fix a set of parameters $c,d,r,n$. To find a formula for $S_{cochar}:= S_{cochar}(c,d,r,n)$, we begin by showing that the solutions to the inhomogenous cocharacter equations fall into equally sized equivalence classes defined by the values $a\in \Z$, and that
\begin{equation*}
S_{cochar}(c,d,r,n) = \frac{A(c,d,r,n)}{n}\cdot  S_{inhom}(c,d,r,n)
\end{equation*}

Characterizing the solutions to the inhomogenous cocharacter equations, we will then provide explicit expressions for $S_{inhom}(c,d,r,n)$ and $A(c,d,r,n)$.

\begin{lemma}
The equation $B_{c,d}\cdot \boldsymbol{x} \equiv a\cdot \boldsymbol{1}_r \pmod{n}$ has a solution if and only if $a$ is a multiple of $A=A(c,d,r,n)$. Thus, $A(c,d,r,n)$ divides $n$.
\label{lem: Bx = A}
\end{lemma}

\begin{proof}
By definition, a solution $\boldsymbol{x}_A$ to the equation $B_{c,d}\cdot \boldsymbol{x} \equiv A\cdot \boldsymbol{1}_r \pmod{n}$ exists. If $a=kA$ for some $k\in \Z$, then $k\boldsymbol{x_A}$ is a solution to $B_{c,d}\cdot \boldsymbol{x} \equiv a\cdot \boldsymbol{1}_r \pmod{n}.$ For the other direction, suppose there exists a positive integer $g$ and solution $\boldsymbol{x}_g\in (\Z/n\Z)^r$ to the equation $B_{c,d}\cdot \boldsymbol{x}_g \equiv g\cdot \boldsymbol{1}_r \pmod{n}$, but that $A$ does not divide $g$. 
Then $jA < g < (j+1) A$ for some positive integer $j$. Therefore,
\begin{align*}
    B_{c,d}\cdot (\boldsymbol{x}_g - j \boldsymbol{x}_A) &\equiv B_{c,d}\cdot  \boldsymbol{x}_g - j B_{c,d}\cdot \boldsymbol{x}_A\\ 
    &\equiv (g - jA)\cdot \boldsymbol{1}_r,
\end{align*}
which contradicts the minimality of $A$. Then, since $\boldsymbol{x} = \boldsymbol{0}_r$ is a solution to $B_{c,d}\cdot  \boldsymbol{x} \equiv n\cdot\boldsymbol{1}_r \equiv \boldsymbol{0}_r \pmod{n}$, the second statement follows.
\end{proof}

Splitting the solutions to \eqref{inhom} into equivalence classes based on $a$, we examine the number of solutions in each class and characterize them more concretely.

\begin{lemma}
For $k\in \{1,...,\frac{n}{A}\}$, let $W_k$ be the set of solutions to $B_{c,d} \cdot \boldsymbol{x} \equiv (kA)\cdot \boldsymbol{1}_r \pmod{n}$. Then $|W_k| = |W_1|$ for all such $k$.
\end{lemma}

\begin{proof}
Consider any $\boldsymbol{x}\in W_1$. The function $\phi_{\boldsymbol{x}}: W_1\rightarrow W_k$ defined by $\boldsymbol{y} \mapsto \boldsymbol{y}+(k-1)\cdot\boldsymbol{x}$ provides a bijection between $W_1$ and $W_k$.
\end{proof}

\begin{lemma}\label{lemma:c-d}
Let $\boldsymbol{x} = (x_1, x_2, \dots, x_r)^T$. Then $\boldsymbol{x}$ solves the inhomogenous cocharacter equations if and only if $cx_1 + dx_j \equiv dx_1 + cx_j \pmod{n}$ for every $2 \leq j \leq r$.
\end{lemma}

\begin{proof}

Let $2 \leq j \leq r$. For a solution $\boldsymbol{x}$, the first row of the equation $B_{c,d}\cdot \boldsymbol{x} \equiv a\cdot \boldsymbol{1}_r \pmod{n}$ tells us that
\begin{equation*}
    c x_1 + d x_j + \sum\limits_{\substack{2 \leq k \leq r\\ k \neq j}} d x_k \equiv a \pmod{n}
\end{equation*}
Subtracting the $j$-th row 
\begin{equation*}
    d x_1 + c x_j + \sum\limits_{\substack{2 \leq k \leq r\\ k \neq j}} d x_k \equiv a \pmod{n},
\end{equation*}
from the first, we obtain
\begin{eqnarray*}
    c x_1 + d x_j &\equiv& d x_1 + c x_j \pmod{n}.
\end{eqnarray*}
For the other direction, suppose $\boldsymbol{x}$ satisfies $cx_1 + dx_j \equiv dx_1 + cx_j \pmod{n}$ for all $j\in \{2,...,r\}$. Then, $\boldsymbol{x}$ satisfies the inhomogeneous cocharacter equations for the value $a \equiv c x_1 + d x_j + \sum\limits_{\substack{2 \leq k \leq r\\ k \neq j}} d x_k \pmod{n}$.
\end{proof}


\begin{prop}
\label{prop: solutions to inhom}
A vector $\boldsymbol{x} = (x_1, x_2, \dots, x_r)^T$ solves the inhomogenous cocharacter equations 
if and only if for each $j\in \{2,...,r\}$ we have $x_j = x_1 + v_j \cfrac{n}{\done}$ for some integer $v_j$ such that $1 \leq v_j \leq \done$.
\end{prop}

\begin{proof}
By Lemma \ref{lemma:c-d}, it suffices to characterize the solutions $\boldsymbol{x}$ to the system of equations given by 
\begin{equation*}
    cx_1 + dx_j \equiv dx_1 + cx_j \pmod{n}
\end{equation*} 
for every $j\in\{2,...,r\}$, or equivalently,
\begin{equation}\label{cx1+dxj=dx1+cxj}
(c - d)(x_1 - x_j) \equiv 0 \pmod{n} .
\end{equation}
Recalling that $\done = \gcd(c-d, n)$, a vector $\boldsymbol{x}$ satisfies \eqref{cx1+dxj=dx1+cxj} exactly when $x_1 - x_j$ is a multiple of $\frac{n}{\done}$ for all $j\in\{2,...,r\}$. Thus, the solutions to the inhomogeneous cocharacter equations are precisely the vectors of the form
\begin{equation*}
    \boldsymbol{x} = x_1 \cdot \boldsymbol{1}_r
    + \frac{n}{\done} (0,v_2,v_3,...,v_r)^T
\end{equation*}
where $0 \leq x_1 < n$ and $v_j \in \Z$ such that $1 \leq v_j \leq \done$ for all $j\in \{2,...,r\}$.
\end{proof}

Now that we have precisely characterized the set of $\boldsymbol{x}$ which solve the inhomogeneous cocharacteristic equations, we can count the size of this set by ranging over all distinct choices of tuples $(x_1,v_2,...,v_r)$, which each yield a distinct solution $\boldsymbol{x}$.

\begin{cor}
The number of solutions to the inhomogenous cocharacter equations is 
\begin{equation*}
S_{inhom}(c,d,r,n) = n \done^{r-1}.
\end{equation*}
\end{cor}

We are now prepared to identify a precise formula for $A$ in terms of $n, r, c,$ and $d$.

\begin{prop}
The minimum positive integer $A$ such that the inhomogenous cocharacter equations have a solution is $A = \gcd\left(\dtwo, \frac{dn}{\done}\right)$, recalling that $\dtwo = \gcd(c+(r-1)d,n)$.
\end{prop}

\begin{proof} Substituting Proposition \ref{prop: solutions to inhom} into \eqref{inhom}, we see that the left-hand side is
\begin{align*}
   \begin{pmatrix}
    c & d & d & \dots & d\\
    d & c & d & \dots & d\\
    d & d & c & \dots & d\\
    \vdots & \vdots & & \ddots & \vdots\\
    d & d & d & \dots & c
    \end{pmatrix}
    \left( x_1 \begin{pmatrix}
    1 \\ 1 \\ 1 \\ \vdots \\ 1
    \end{pmatrix}
    + \frac{n}{\done} \begin{pmatrix}
    0 \\ v_2 \\ v_3 \\ \vdots \\ v_r
    \end{pmatrix}
    \right)
    &\equiv x_1 (c + (r-1)d)\begin{pmatrix}
    1 \\ 1 \\ 1 \\ \vdots \\ 1
    \end{pmatrix} + \frac{n}{\done}
    \begin{pmatrix}
    d v_2 + d v_3 + \dots + d v_r \\ c v_2 + d v_3 + \dots + d v_r \\ d v_2 + c v_3 + \dots + d v_r \\ \vdots \\ d v_2 + d v_3 + \dots + d v_{r-1} + c v_r
    \end{pmatrix}.
\end{align*}
To have a solution, every row of this expression must must equal a constant $A$. Looking at the first row,
\begin{align*}
    A &\equiv x_1 (c + (r - 1) d) + \frac{dn}{\done}(v_2 + v_3 + \dots + v_r) \pmod{n}.
\end{align*}
From the proof of Proposition \ref{prop: solutions to inhom}, $x_1$ and $v_2 + v_3 + \dots + v_r$ are both arbitrary constants. Thus, the minimum value $A$ can have is $\gcd\left(c + (r - 1) d, \frac{dn}{\done}, n\right) = \gcd\left(\dtwo,\frac{dn}{\done}\right)$.
\end{proof}


 Note, since $\done$ divides $c-d$, we can equivalently write $A$ as $A = \gcd\left(\dtwo, \frac{cn}{\done}\right) = \gcd\left(\dtwo, \frac{dn}{\done}\right) = \gcd\left(\dtwo,\frac{n}{\done}\gcd(c,d,n)\right)$. Thus, we arrive at a closed form for $S_{cochar}$ in terms of $c,d,r,n$.

\begin{thm}
The number of solutions to the cocharacter equations is
\begin{align*}
S_{cochar}(c,d,r,n) = \done^{r-1} \gcd\left(\dtwo, \frac{dn}{\done}\right).
\end{align*}
\label{thm: mainthm}
\end{thm}

\section{Coroot Diophantine Equations}\label{Coroot}

Inspired by the constants $c+(r-1)d$ and $c-d$ showing up in the cocharacter equations, we define a second system of related equations more closely tied to the root structure of $GL_r(F)$.

\begin{definition}
The \emph{coroot equations} are the system of $r$ equations:
\begin{align*}
    (c-d)(x_i - x_r)  &\equiv 0 \pmod{n} \hspace{0.5cm} \text{for all $i\in\{1,...,r-1\}$,} \label{d2-coroot}\\
    (c+(r-1)d)(x_1 + \cdots + x_r) & \equiv 0 \pmod{n}. 
\end{align*}
\end{definition}

We call these the coroot equations because the $i$-th equation arises from evaluating $\boldsymbol{x}\in Y$ against the coroot $\varepsilon_i^\vee - \varepsilon_r^\vee$ in the bilinear form $B_{c,d}$ for $i\in\{1,...,r-1\}$. We could similarly evaluate against the simple coroots $\varepsilon_i^\vee - \varepsilon_{i+1}^\vee$, but this formulation will be more useful for our purposes. 

In some cases, the coroot and cocharacter equations are equivalent, but in other cases they are not: counting the solutions to the coroot equations and examining this connection will give us an alternate formula for $S_{cochar}$. 

\begin{remark}
This system also illuminates the difference between metaplectic covers of $SL_r(F)$ and $GL_r(F)$. For any cocharacter $\boldsymbol{x}$ for $SL_r(F)$, the last coroot equation is vacuously true, since $x_1 + \cdots +x_r \equiv 0 \pmod{n}$ is necessary for the resulting matrices $\boldsymbol{x}(a)$ to have determinant one for any $a\in F$. In this case, the cocharacter and coroot equations are equivalent, and they both give $|\Lambda_{fin} \cap SL_r(F)| = \done^{r-1}$. 
\end{remark}

\begin{thm}
The number of solutions to the coroot equations is
\begin{align*}
S_{coroot}(c,d,r,n) = \done^{r-1} \dtwo \gcd\left(\frac{n}{\done}, \frac{n}{\dtwo}, r\right).
\end{align*}
\label{thrm:coroot solutions}
\end{thm}

As in Section \ref{SolvingCochar}, we prove general properties about the solutions to the coroot equations. These descriptions will allow us to directly relate $S_{cochar}$ to $S_{coroot}$ in Section \ref{Comparisons}. 

\begin{proof}
We start with a change of variables. Consider the coroot system in variables $y_1,...,y_{r-1},z$ written as
\begin{align*}
    (c-d)y_i  &\equiv 0 \pmod{n} \hspace{0.5cm} \text{for all $i\in\{1,...,r-1\}$,}\\
    (c+(r-1)d)z & \equiv 0 \pmod{n}.
\end{align*}
In terms of these variables, there are $\done^{r-1} \cdot \dtwo$ tuples $(y_1,...,y_{r-1},z)$ that solve the coroot equations: $y_i$ are all multiples of $\frac{n}{\done}$ and $z$ is a multiple of $\frac{n}{\dtwo}$. Let $S_{Y,Z}$ be the set of such tuples.

We then classify $\boldsymbol{x}$ satisfying $y_i= x_i-x_r$ and $z= x_1 + \cdots + x_r$ such that $(y_1,..,y_{r-1},z)\in S_{Y,Z}$: that is, the set of $\boldsymbol{x}$ satisfying the original formulation of the coroot equations. Note that $x_i = y_i+x_r$, so rearranging the final coroot equation, we have
\begin{equation}
    rx_r \equiv z - (y_1 + \cdots + y_{r-1})\pmod{n}, \label{rx=a}
\end{equation}
and thus the number of solutions in terms of $\boldsymbol{x}$ versus in terms of $(y_1,...,y_{r-1},z)$ depends on whether $r$ is invertible mod $n$. Let $b = \gcd(n, r)$. Then $x_r$ has $b$ solutions when $z - (y_1+\cdots + y_{r-1})$ is a multiple of $b$ and no solutions otherwise. A straightforward calculation verifies that there is no overlap between the sets of $\boldsymbol{x}$ for distinct tuples $(y_1,..,y_{r-1},z)\in S_{Y,Z}.$

Let $Fr_b(\done,\dtwo, n)$ be the proportion of $(y_1, \dots, y_r, z)$ tuples that will yield a valid solution to the coroot equations. In other words,
   $$Fr_b := \frac{|\{(y_1, \cdots, y_r, z) \in S_{Y,Z}: z- y_1 - \cdots -y_r \text{ is a multiple of } b\}|}{|S_{Y,Z}|}. $$
Then $S_{coroot}(c,d,r,n) = \done^{r-1} \dtwo \cdot b \cdot Fr_b(\done, \dtwo,n),$ and it will suffice to develop a formula for $Fr_b(\done, \dtwo,n)$.
\end{proof}

\begin{remark}
\label{rmk: Fr_1}
When $n$ and $r$ are relatively prime, $r$ is invertible. Thus  $Fr_b$ evaluates to $1$ because any tuple we pick adds to a multiple of $b=1$, so in this case the two sets of variables give equivalent conditions and $S_{coroot} = |S_{Y,Z}|$.
\end{remark}

\begin{prop}
The function $Fr_b$ evaluates to
\begin{eqnarray*}
   Fr_b(\done,\dtwo,n) = \frac{1}{b}\cdot \gcd\left(\frac{n}{\done}, \frac{n}{\dtwo}, b\right).
\end{eqnarray*}
\end{prop}

\begin{proof}
We proceed by carefully considering the overlaps of factors of $b$ with those of $\frac{n}{\done}, \frac{n}{\dtwo}.$ Let $k_1 = \gcd(\frac{n}{\done},b)$ and $m_1\in \Z$ such that $b = m_1k_1$. Similarly, let $k_2 = \gcd(\frac{n}{\dtwo},b)$ and $m_2\in \Z$ such that $b = m_2 k_2$.

Since $y_i$ is a multiple of $\frac{n}{\done}$, it is also a multiple of $k_1$: examining which multiples are possible modulo $b$, we see that $y_i \pmod{b}$ can be any of the $m_1$ multiples of $k_1$ in $\Z/b\Z$ with equal probability.
Similarly, considering the sum $y = \sum_{i=1}^{r-1}y_i$, we claim that the same is true for $y$. Let $1\leq g\leq m_1$ and suppose $y\equiv gk_1\pmod{b}$: if we pick any arbitrary $y_1, y_2, \dots, y_{r-2}$, we are left with
\begin{eqnarray*}
y_{r-1} \equiv gk_1 - \sum_{i=1}^{r-2}y_i \pmod{b}.
\end{eqnarray*}
The right-hand side of this equation defines some equivalence class $ \ell k_1 \pmod{b}$ from which we must choose $y_{r-1}$ to ensure that $y \equiv gk_1 \pmod{b}$. Exactly $\frac{1}{m_1}$ of the possible values of $y_{r-1}$ place us in the correct equivalence class for a given $g$. Thus, $y$ falls into the equivalence classes $k_1, 2k_1, \dots, m_1k_1$ with equal probability. 

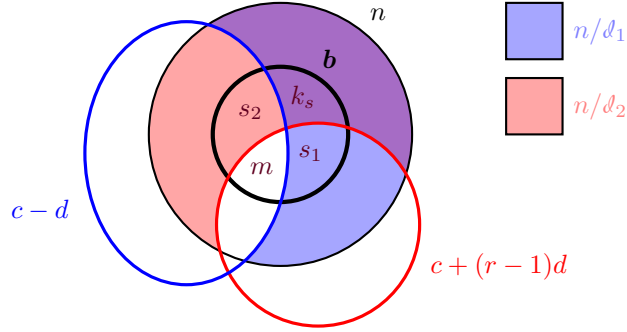
\begin{figure}[h]
    \centering
    \begin{tikzpicture}
    \scope
        \clip (-2,-2) rectangle (2,2)
        (0.5,-1.2) circle (1.35);
        \fill[fill=red, fill opacity=0.35] (0,0) circle (1.75);
    \endscope
    \scope
        \clip (-2,-2) rectangle (2,2)
        (-1.25,-0.25) ellipse (1.35 and 1.75);
        \fill[fill=blue, fill opacity=0.35] (0,0) circle (1.75);
    \endscope
    \draw[color=black, thick](0,0) circle (1.75);
    \draw[color=black, ultra thick](0,0) circle (0.9);
    \draw[color=red, very thick](0.5,-1.2) circle (1.35);
    \draw[color=blue, very thick](-1.25,-0.25) ellipse (1.35 and 1.75);
    \filldraw[thick, fill=blue, fill opacity=0.35] (3,1) rectangle (3.75,1.75);
    \filldraw[thick, fill=red, fill opacity=0.35] (3,0) rectangle (3.75,.75);
    \node [purple!50!black] at (0.3,0.5){$k_s$};
    \node [purple!50!black] at (0.4,-0.25){$s_1$};
    \node [purple!50!black] at (-0.4,0.3){$s_2$};
    \node [purple!50!black] at (-0.25,-0.45){$m$};
    \node [blue] at (-3.2,-1){$c-d$};
    \node [red] at (2.9,-1.75){$c+(r-1)d$};
    \node [thick, black] at (0.65,1){$\boldsymbol{b}$};
    \node [black] at (1.3,1.6){$n$};
    \node [blue, opacity=0.5] at (4.25,1.375){$n/\done$};
    \node [red, opacity=0.5] at (4.25,.375){$n/\dtwo$};
    \end{tikzpicture}
    \caption{A visualization of our factorization of $b = \gcd(r,n)$, where the overlap of any circle or shaded region with the circle for $b$ contains a factorization for their greatest common divisor. Note that the purple region is the overlap of the red and blue regions.}
    \label{VennDiagram}
\end{figure}

Likewise, $z \pmod{b}$ can be any of the $m_2$ multiples of $k_2$ modulo $b$ with equal probability. To interface between $y$ and $z$, we must factor further: let $k_s = \gcd(k_1,k_2)$ so that $k_1 = s_1 k_s$ and $k_2 = s_2 k_s$, and $\gcd(s_1,s_2)=1$. Letting $m = \gcd(m_1, m_2)$, factor $b$ completely as $b = m s_1 s_2 k_s$. The reader may find it helpful to refer to Figure \ref{VennDiagram}, which provides a visualization of how this factorization relates $b, \frac{n}{\done},$ and $\frac{n}{\dtwo}$.

We now identify the proportion of $y$- and $z$-values that satisfy $z - y \equiv 0 \pmod{b}$. Since $y,z,b$ all contain a factor of $k_s,$ let $y = \alpha k_1 = \alpha s_1 k_s$ and $z = \beta k_2 = \beta s_2 k_s$ for some $\alpha,\beta\in \Z$. Then, equivalently, we seek the proportion of $(\alpha,\beta)$ pairs such that
\begin{align*}
   \beta s_2 &\equiv \alpha s_1 \pmod{m s_1 s_2}.
\end{align*}
Since $s_1,s_2$ are coprime, we must have $\alpha = a s_2$ for some $a\in\Z$. Exactly $\frac{1}{s_2}$ of the possible $\alpha$-values are multiples of $s_2$. Then
\begin{align*}
   \beta s_2 &\equiv a s_1 s_2 \pmod{m s_1 s_2}
\end{align*}
which has solutions only for $\beta \equiv a s_1 \pmod{b}$. Out of the $m_2 = m s_1$ equivalence classes that $z$ can fall into, only the one defined by $\beta = a s_1$ works. Therefore, 
\begin{align*}
  Fr_b(\done,\dtwo,n) &=  \frac{1}{s_2}\left(\frac{1}{ms_1}\right) = \frac{1}{ms_1 s_2} = \frac{k_s}{b} = \frac{\gcd\left(\frac{n}{\done}, \frac{n}{\dtwo}, b\right)}{b}. \qedhere
\end{align*}
\end{proof}

Since $b = \gcd(r,n)$, we have $\gcd\left(\frac{n}{\done}, \frac{n}{\dtwo}, b\right)=  \gcd\left(\frac{n}{\done},\frac{n}{\dtwo},r\right)$, which completes proof of Theorem \ref{thrm:coroot solutions} and allows us to express $S_{coroot}$ solely in terms of $c,d,r,$ and $n$.

\section{Proof of Main Theorem 1 Part 2}\label{Comparisons}

We are now ready to prove the second part of our main result. In this section, we show how the coroot equations are obtained from the cocharacter equations, and how this relates $S_{coroot}$ and $S_{cochar}$.

\begin{thm}\label{thm:Main2}
The number of solutions to the cocharacter equations can also be defined as
$$S_{cochar} = S_{coroot} \cdot  \frac{1}{n}\cdot \emph{\lcm}\left(\gcd\left(\dtwo, \frac{dn}{\done}\right), \frac{n}{\gcd(n,r)}\right).$$
\end{thm}

Let $M(c,d,r,n):=\text{\lcm}\left(\gcd\left(\dtwo, \frac{dn}{\done}\right), \frac{n}{\gcd(n,r)}\right).$ We will also prove the following formula for $M(c,d,r,n)$, which will be useful for our investigation into special dimensions for $\mathfrak{W}$ in Section \ref{Structure}.

\begin{prop}
Let $r,n$ have prime factorizations $r = p_1^{\ell_1} p_2^{\ell_2} \dots p_j^{\ell_j}$ and $n = p_1^{m_1} p_2^{m_2} \dots p_j^{m_j}$. For every $1 \leq i \leq j$, let
$$(c-d) \equiv c_i p_i^{s_i} \pmod{p_i^{m_i}}
\text{ and } d \equiv d_i p_i^{t_i} \pmod{p_i^{m_i}}
\text{ for each $1 \leq i \leq j$}$$
so that $0 \leq s_i, t_i \leq m_i$. Let $\mu_i = \min(m_i, \ell_i)$ and $c_i, d_i$ are relatively prime to $p_i$. Then
\begin{equation*} 
M(c,d,r,n) = \prod_{i = 1}^j p_i^{\max(m_i - \mu_i, \min(s_i, t_i + m_i - s_i))}.
\label{M def as primes}
\end{equation*}
\label{prop: M def as primes}
\end{prop}

 To obtain the coroot equations from the cocharacter equations, we can multiply the matrix $B_{c,d}$ which defines the cocharacter equations by the $r\times r$ matrix
\[L_r =  \begin{pmatrix}
    1 & 0 & \dots & 0 & -1\\
    0 & 1 & \dots & 0 & -1\\
    \vdots & \vdots & \ddots & & \vdots\\
    0 & 0  & & 1 & -1\\
    1 & 1 &  \dots & 1 & 1
    \end{pmatrix}.\]
That is, the new system of equations specified by this transformation is
\begin{equation*}
    L_r B_{c,d}\cdot \boldsymbol{x} \equiv \boldsymbol{0}_r \pmod{n}
\end{equation*}
which gives us precisely the coroot equations. Likewise, multiplying the coroot equations by 
\[L'_r= \begin{pmatrix}
    r-1 & -1 & \dots & -1 & 1\\
    -1 & r-1 & \dots & -1 & 1\\
    \vdots & \vdots & \ddots & & \vdots\\
    -1 & -1 & & r-1 & 1\\
    -1 & -1 & \dots & -1 & 1
    \end{pmatrix}\]
obtains the cocharacter equations multiplied by $r$. That is,
\[L_r' (L_rB_{c,d})\cdot\boldsymbol{x} \equiv r\cdot B_{c,d}\cdot\boldsymbol{x} \equiv \boldsymbol{0}_r \pmod{n}. \]

As we discussed in Remark \ref{rmk: Fr_1}, if $r$ and $n$ are relatively prime, then $r$ has an inverse $r^{-1}$ in $\Z/n\Z$ and the cocharacter and coroot equations are equivalent. However, if $r$ is not invertible modulo $n$, going back from the coroot equations to the cocharacter equations is more complicated.

Recall that $b = \gcd(r,n)$. Then $r\cdot B_{c,d}\boldsymbol{x} \equiv \boldsymbol{0}_r \pmod{n}$ factors into
\begin{equation*}
    b\left(\frac{r}{b}\right) \begin{pmatrix}
    c & d & d & \dots & d\\
    d & c & d & \dots & d\\
    d & d & c & \dots & d\\
    \vdots & \vdots & & \ddots & \vdots\\
    d & d & d & \dots & c
    \end{pmatrix} \begin{pmatrix}x_1 \\ x_2 \\ x_3 \\ \vdots \\ x_r
    \end{pmatrix} 
    \equiv \begin{pmatrix}
    0 \\ 0 \\ 0 \\ \vdots \\ 0
    \end{pmatrix}\hspace{0.25cm} \left(\text{mod } b\cdot \frac{n}{b} \right).
\end{equation*}
Since both sides of the equation and the modulus are multiples of $b$, this implies that
\begin{equation*}
    \left(\frac{r}{b}\right) B_{c,d}\cdot  \boldsymbol{x} 
    \equiv \boldsymbol{0}_r\hspace{0.25cm} \left(\text{mod } \frac{n}{b} \right).
\end{equation*}
The number $\frac{r}{b}$ is relatively prime to $\frac{n}{b}$ and therefore invertible in $\Z/{\frac{n}{b}}\Z$, so
\begin{equation}
    B_{c,d} \cdot \boldsymbol{x} 
    \equiv \boldsymbol{0}_r \hspace{0.25cm} \left(\text{mod } \frac{n}{b} \right).
    \label{Bx = 0 modn/b}
\end{equation}

\begin{remark}
Again, \eqref{Bx = 0 modn/b} shows that the coroot and cocharacter equations are equivalent when $r$ and $n$ are relatively prime, since then $n=n/b$ and \eqref{Bx = 0 modn/b} recovers exactly the cocharacter equations.
\end{remark}

If we are not in that case, i.e., if $b \neq 1$, then for any coroot solution $\boldsymbol{x}$, we get 
\begin{equation*}
B_{c,d} \boldsymbol{x} \equiv \frac{n}{b} \boldsymbol{v} \pmod{n}
\end{equation*}
for some vector $\boldsymbol{v} \in \Z^r$. We now show that each solution to the coroot equations also satisfies inhomogeneous cocharacter equations for particular values of $a$. 

\begin{lemma}
   The coroot equations are equivalent to the inhomogeneous cocharacter equations with the condition that $a\in (n/b)\Z$.
\end{lemma}

\begin{proof}
Let $\boldsymbol{x}$ be a solution to the coroot equations and let the $i$th row of the left-hand side of Equation \eqref{Bx = 0 modn/b} be $$w_i = cx_i + \sum_{j \neq i} dx_j.$$ Then by definition,
\begin{align*}
    w_i - w_r &= (c - d)(x_i - x_r) \equiv 0 \pmod{n}.
\end{align*}
Thus, for some $k\in \{1,...,b\}$,
\begin{equation}\label{eqn:inhom a=nk/b}
 B_{c,d} \boldsymbol{x} \equiv \frac{n}{b} k \cdot \boldsymbol{1}_r \pmod{n}.\end{equation}

Similarly, suppose $\boldsymbol{x}$ satisfies \eqref{eqn:inhom a=nk/b}. Then defining $w_i$ as above,
\begin{align*}
(c-d)(x_i - x_r) &\equiv w_i - w_r \equiv k\frac{n}{b} - k\frac{n}{b} \equiv 0 \pmod{n}\\
(c+(r-1)d)(x_1 + x_2 + \dots x_r) &\equiv w_1 + w_2 + \dots + w_r \equiv r\left(k\frac{n}{b}\right) \equiv 0 \pmod{n}. \qedhere
\end{align*}
\end{proof}

In Section \ref{SolvingCochar} we showed that Equation \eqref{eqn:inhom a=nk/b} has solutions if and only if $k\frac{n}{b}$ is a multiple of $A(c,d,r,n)$, and that each class of solutions (defined by having the same $k$) is of the same size. As in that section, we want to find the smallest nonzero $k$ for which \eqref{eqn:inhom a=nk/b} has a solution.

\begin{definition}
Let $\kappa(c,d,r,n)$ be the smallest positive value of $k$ such that there is a solution to Equation \eqref{eqn:inhom a=nk/b}. 
\end{definition}

Again, it will often be clear from context that we are working with a particular fixed $c,d,r,n$ in which case we will write $\kappa$ and $M$ for brevity instead of $\kappa(c,d,r,n)$ and $M(c,d,r,n)$.

\begin{proof}[Proof of Theorem \ref{thm:Main2}]
We can relate the values of $\kappa(c,d,r,n)$ and $A(c,d,r,n)$ as follows:

\begin{equation*}
\kappa \cdot \frac{n}{b} = \text{lcm}\left(A, \frac{n}{b}\right) = \text{lcm}\left(\gcd\left(\dtwo, \frac{d n}{\done} \right), \frac{n}{b}\right).
\end{equation*}
Let $M(c,d,r,n):= \text{lcm}\left(\gcd\left(\dtwo, \frac{d n}{\done} \right), \frac{n}{b}\right).$ Then there are $\frac{n}{M} = \frac{b}{\kappa}$ equivalence classes of solutions to Equation \eqref{eqn:inhom a=nk/b}. Exactly one of these equivalence classes---the one given by $k = b$---gives the solutions to the cocharacter equations. Therefore,
\begin{equation*}
    S_{cochar} = S_{coroot} \cdot \frac{M}{n} = S_{coroot} \cdot \frac{\kappa}{b}.
    \label{Scoch = M/n Scort}
\end{equation*}

Substituting in our earlier expressions for the values of $S_{coroot}$ and $M$, we obtain
\begin{equation*}
    S_{cochar}=\frac{\done^{r-1} \dtwo}{n} \gcd\left(\frac{n}{\done}, \frac{n}{\dtwo}, r\right) \text{lcm}\left(\frac{n}{\gcd(r,n)}, \gcd\left(\dtwo, \frac{d n}{\done} \right)\right). \qedhere
\end{equation*}
\end{proof}

Although it is not immediately clear from looking at this equation, this formula is equivalent to the one given in Theorem \ref{thm: mainthm}. One area of future work would be to simplify this expression and more directly understand why it is equivalent to the statement of Theorem \ref{thm: mainthm}. Furthermore, while this formula appears more complicated than that of Theorem \ref{thm: mainthm}, this approach is perhaps more suitable for extending past $GL_r(F)$, as the metaplectic Whittaker functions developed in Section \ref{MWhit} can be defined over any reductive group and this approach is more closely related to the root data structure of reductive groups.

Using the same visualization tables we used in Section \ref{Cochar} for $S_{cochar}$ shows more directly how $M$ and $\kappa$ change as we vary $c,d,r,$ and $n$. Here, for a fixed $r,n$, let the entry in position $(c,d)$ be $\kappa(c,d,r,n)$. (To achieve a matching table for $M$, multiply the $\kappa$ table by $n/b$.) 

\begin{example} For $n = 8 = 2^3$ and $r = 2^\ell$, the following tables show how $\kappa$ changes as $\ell$ increases from $1$ to $3$.  Because $M$ and $\kappa$ depend on $\mu = \min(\ell, m)$ rather than on $\ell$, any $\kappa$ table for $\ell >3$ would be identical to the table for $\ell = 3$.

\begin{center}
\scalebox{0.8}{
\begin{tikzpicture}
\node(0,0){$\small\begin{matrix}
    2 & 1 & 1 & 1 & 1 & 1 & 1 & 1\\
    1 & 1 & 1 & 1 & 1 & 1 & 1 & 1\\
    1 & 1 & 1 & 1 & 1 & 1 & 1 & 1\\
    1 & 1 & 1 & 1 & 1 & 1 & 1 & 1\\
    1 & 1 & 1 & 1 & 1 & 1 & 1 & 1\\
    1 & 1 & 1 & 1 & 1 & 1 & 1 & 1\\
    1 & 1 & 1 & 1 & 1 & 1 & 1 & 1\\
    1 & 1 & 1 & 1 & 1 & 1 & 1 & 1
    \end{matrix}$};
\node [xshift =5cm](0,0){$\small\begin{matrix}
     4 & 1 & 1 & 1 & 2 & 1 & 1 & 1\\
     1 & 1 & 1 & 1 & 1 & 1 & 1 & 1\\
     1 & 1 & 1 & 1 & 1 & 1 & 2  & 1\\
     1 & 1 & 1 & 1 & 1 & 1 & 1 & 1\\
     2 & 1 & 1 & 1 & 2 & 1 & 1 & 1\\
     1 & 1 & 1 & 1 & 1 & 1 & 1 & 1\\
     1 & 1 & 2 & 1 & 1 & 1 & 1 & 1\\
     1 & 1 & 1 & 1 & 1 & 1 & 1 & 1
     \end{matrix}$};
\node[xshift =10cm] (0,0){$\small\begin{matrix}
    8 & 1 & 2 & 1 & 4 & 1 & 2 & 1\\
    1 & 1 & 1 & 2 & 1 & 2 & 1 & 2\\
    2 & 1 & 2 & 1 & 2 & 1 & 4  & 1\\
    1 & 2 & 1 & 1 & 1 & 2 & 1 & 2\\
    4 & 1 & 2 & 1 & 4 & 1 & 2 & 1\\
    1 & 2 & 1 & 2 & 1 & 1 & 1 & 2\\
    2 & 1 & 4 & 1 & 2 & 1 & 2 & 1\\
    1 & 2 & 1 & 2 & 1 & 2 & 1 & 1
    \end{matrix}$};

\node[yshift=-2.2cm] {$n = 8, r = 2$};
\node[xshift=5cm, yshift=-2.2cm] {$n = 8, r = 4$};
\node[xshift=10cm, yshift=-2.2cm] {$n = 8, r = 8$};

\draw[red, thick] (-2,1.6) -- (2,-1.6);  
\draw[red, thick,xshift=5cm] (-2,1.6) -- (2,-1.6);  
\draw[red, thick,xshift=10cm] (-2,1.6) -- (2,-1.6);  
\draw[blue, thick, xshift=5cm] (0.1,1.6) -- (2,.1); 
\draw[blue, thick, xshift=5cm] (-2,0) -- (0,-1.6); 
\draw[blue, thick, xshift=10cm] (0.1,1.6) -- (2,.1); 
\draw[blue, thick, xshift=10cm] (-2,0) -- (0,-1.6); 
\draw[LimeGreen, thick, xshift=10cm] (1.1,1.6) -- (2,.9); 
\draw[LimeGreen, thick, xshift=10cm] (-.9,1.6) -- (2,-.7);
\draw[LimeGreen, thick, xshift=10cm] (-2,.8) -- (1,-1.6); 
\draw[LimeGreen, thick, xshift=10cm] (-2,-.8) -- (-1,-1.6);
\end{tikzpicture}
}
\end{center}
The entries in these tables are determined by the main diagonals they lie on, which are described by $s$, and the columns they lie in, which are described by $t$ and index how far down the main diagonal an entry is. In particular, notice that the only difference between the matrices for $\ell$ and $\ell+1$ is that a specific fraction of the elements on each of the diagonals in the latter matrix have been multiplied by 2. For example, for $\ell=2$, this fraction is 1/4 for the red diagonal and 1/2 for the blue.
\end{example}

These tables are also useful for visualizing the effect of combining distinct primes. 



\begin{example}
When $r = 2^2$ and $n = 2^2\cdot 3$, we see that the table for $\kappa$ is a $3\times 3$ tessellation of that for $r=2^2, n=2^2$:
\begin{equation*}
\begin{tikzpicture}[baseline=-1pt]
\node at (0,0){$\begin{matrix}
    4 & 1 & 2 & 1\\
    1 & 1 & 1 & 2\\
    2 & 1 & 2 & 1\\
    1 & 2 & 1 & 1
\end{matrix}$};
\node at (0,-1.5){$r = 2^2, n = 2^2$};
\end{tikzpicture}  
\hspace{2cm}
\begin{tikzpicture}[baseline=0pt]
\node at (0,0){$\begin{array}{cccc|cccc|cccc}
    4 & 1 & 2 & 1 & 4 & 1 & 2 & 1 & 4 & 1 & 2 & 1\\
    1 & 1 & 1 & 2 & 1 & 1 & 1 & 2 & 1 & 1 & 1 & 2\\
    2 & 1 & 2 & 1 & 2 & 1 & 2 & 1 & 2 & 1 & 2 & 1\\
    1 & 2 & 1 & 1 & 1 & 2 & 1 & 1 & 1 & 2 & 1 & 1\\
    \hline
    4 & 1 & 2 & 1 & 4 & 1 & 2 & 1 & 4 & 1 & 2 & 1\\
    1 & 1 & 1 & 2 & 1 & 1 & 1 & 2 & 1 & 1 & 1 & 2\\
    2 & 1 & 2 & 1 & 2 & 1 & 2 & 1 & 2 & 1 & 2 & 1\\
    1 & 2 & 1 & 1 & 1 & 2 & 1 & 1 & 1 & 2 & 1 & 1\\
    \hline
    4 & 1 & 2 & 1 & 4 & 1 & 2 & 1 & 4 & 1 & 2 & 1\\
    1 & 1 & 1 & 2 & 1 & 1 & 1 & 2 & 1 & 1 & 1 & 2\\
    2 & 1 & 2 & 1 & 2 & 1 & 2 & 1 & 2 & 1 & 2 & 1\\
    1 & 2 & 1 & 1 & 1 & 2 & 1 & 1 & 1 & 2 & 1 & 1
\end{array}$};
\node at (0,-3){$r = 2^2, n = 12 = 3 \cdot 2^2$};
\end{tikzpicture}
\end{equation*}
\end{example}

\begin{example} Below are the $\kappa$ tables for $n = 6$ and $r = 2,3,6,$ respectively. Note that the table for $r = 6$ is obtained by multiplying the tables for $r = 2$ and $r = 3$ together elementwise. Upon proving Proposition \ref{prop: M def as primes}, we will see that this is true in greater generality.
\begin{equation*}
\begin{tikzpicture}
    \node at (0,0){$\begin{array}{cc|cc|cc}
     2 & 1 & 2 & 1 & 2 & 1\\
     1 & 1 & 1 & 1 & 1 & 1\\
     \hline
     2 & 1 & 2 & 1 & 2 & 1\\
     1 & 1 & 1 & 1 & 1 & 1\\
     \hline
     2 & 1 & 2 & 1 & 2 & 1\\
     1 & 1 & 1 & 1 & 1 & 1
     \end{array}$};
     \node at (0,-1.75){$r = 2, n = 6$};
\end{tikzpicture}
\hspace{1cm}
\begin{tikzpicture}
    \node at (0,0){$\begin{array}{ccc|ccc}
     3 & 1 & 1 & 3 & 1 & 1\\
     1 & 1 & 1 & 1 & 1 & 1\\
     1 & 1 & 1 & 1 & 1 & 1\\
     \hline
     3 & 1 & 1 & 3 & 1 & 1\\
     1 & 1 & 1 & 1 & 1 & 1\\
     1 & 1 & 1 & 1 & 1 & 1
     \end{array}$};
     \node at (0,-1.75){$r = 3, n = 6$};
\end{tikzpicture}
\hspace{1cm}
\begin{tikzpicture}
    \node at (0,0){$\begin{matrix}
     6 & 1 & 2 & 3 & 2 & 1\\
     1 & 1 & 1 & 1 & 1 & 1\\
     2 & 1 & 2 & 1 & 2 & 1\\
     3 & 1 & 1 & 3 & 1 & 1\\
     2 & 1 & 2 & 1 & 2 & 1\\
     1 & 1 & 1 & 1 & 1 & 1\\
     \end{matrix}$};
     \node at (0,-1.75){$r = 6, n = 6$};
\end{tikzpicture}.
\end{equation*}
\end{example}

Expressing the values of $M$ and $\kappa$ directly in terms of the prime factors of $n,r,c$ and $d$, we are ready to prove Proposition \ref{prop: M def as primes}.

\begin{proof}[Proof of Proposition \ref{prop: M def as primes}]
Let $r = p_1^{\ell_1} p_2^{\ell_2} \dots p_j ^{\ell_j}$ and $n = p_1^{m_1} p_2^{m_2} \dots p_j ^{m_j},$ where some of $\ell_i$ or $m_i$ may be 0. Recalling the definition of $\dtwo,$ we have that
\[ M = \text{lcm} \left(\gcd\left(c+(r-1)d, n, \frac{dn}{\done}\right), \frac{n}{b}\right).\]
Thus, the only prime factors of $M$ are those that are prime factors of $n$, and $M$ is multiplicative over these prime powers. Considering only the power of $p_i$ arising in $M$ for some $i \in \{1,j\}$, let
$$(c-d) \equiv c_i p_i^{s_i} \pmod{p_i^{m_i}}
\text{ and } d \equiv d_i p_i^{t_i} \pmod{p_i^{m_i}}
\text{ for each $1 \leq i \leq j$}$$
so that $0 \leq s_i, t_i \leq m_i$, and $c_i,d_i$ are relatively prime to $p_i$. Let $\mu_i = \min(m_i, \ell_i)$. Then, $n/b = b_i p_i^{m_i - \mu_i}$, where $b_i$ is relatively prime to $p_i$.


Then consider the power of $p_i$ arising from $\gcd\left(c+(r-1)d, n, \frac{dn}{\done}\right)$: recalling that $\done = \gcd(c-d,n)$, the power of $p_i$ in $\done$ is $\min\{s_i, m_i\} = s_i$. Then, the power of $p_i$ in $dn/\done$ is $t_i+m_i - s_i$. It remains to consider the power arising in the first component of the gcd.

Consider $c+ (r-1)d = c-d + r\cdot d$. Substituting in the factorizations, we have
\[ c+ (r-1)d \equiv c_i p_i^{s_i} + r_i d_i p_i^{t_i + \mu_i} \pmod{p_i^{m_i}}, \]
where $r_i$ is relatively prime to $p_i$.

We consider three cases. If $s_i < t_i + \mu_i$, the power of $p_i$ in this component is $s_i$. Thus, the power of $p_i$ in the gcd is $\min\{s_i, t_i+m_i-s_i\}$, since $s_i\leq m_i$. The power of $p_i$ in $M$ is then
\[ \max\{m_i-\mu_i, \min\{s_i, t_i+m_i-s_i\}\}.\]

Next suppose that $s_i > t_i + \mu_i$. Then we have that the power of $p_i$ in the gcd is $\min\{t_i+\mu_i, t_i+m_i-s_i\}$. However, then the power of $p_i$ in $M$ is
\begin{align*}
   \max\{m_i-\mu_i, \min\{t_i+\mu_i, t_i+m_i-s_i\}\} = m_i-\mu_i,
\end{align*}
since $s_i > t_i + \mu_i$, so $t_i + m_i - s_i < m_i - \mu_i$. 

Lastly, if $s_i = t_i + \mu_i$, then
\[ c-d + rd \equiv p_i^{s_i} (c_i + r_id_i) \pmod{p_i^{m_i}}.\]
Note that $c_i + r_i d_i$ may create an additional factor $p_i^{\tau_i}$ for some integer $\tau_i\geq 0$. Then the power of $p_i$ appearing in the gcd is 
\begin{align*}
    \min\{s_i+\tau_i, m_i, t_i + m_i-s_i\} = \min\{s_i+\tau_i, m_i- \mu_i\}
\end{align*}
Then the power of $p_i$ in $M$ is 
\begin{align*}
    \max\{ \min\{s_i+\tau_i, m_i- \mu_i\}, m_i-\mu_i\} = m_i - \mu_i.
\end{align*}
Collecting the three cases together, the expression 
\[ \max\{m_i - \mu_i, \min\{s_i, t_i + m_i - s_i\}\}\]
matches the power of $p_i$ in $M$ in each case, completing the proof of the proposition.
\end{proof}

\begin{cor} The quantity $\kappa(c,d,r,n)$ is multiplicative over powers of distinct primes.
\end{cor}

In the next section, we will see that the two different approaches for $S_{cochar}$ are each useful in different ways. One potentially fruitful avenue for future exploration would be to see precisely why these two formulae are equal, as it is not easily apparent. As the second approach relates more directly to the root structure of $GL_r$ as a reductive group, but the first approach yields a simpler formula and proof, this connection would illuminate a way to extend the simpler formula to general reductive groups.

\section{Structure of the Whittaker Space}\label{Structure}

These investigations into the structure of $\Lambda_{fin}$ not only give us a method of calculating $\dim(\mathfrak{W})$, they also illuminate how the parameters $c,d,r,$ and $n$ affect the structure of $\mathfrak{W}$ in different ways. In this section, we start with a few natural corollaries to both parts of Main Theorem 1 (Theorems \ref{thm: mainthm} and \ref{thm:Main2}) and discuss how they relate to the literature. We then develop necessary and sufficient conditions for $\dim(\mathfrak{W})$ to be of maximum and minimum dimension, as well as the conditions for several other desirable dimensions for further connections.

\begin{cor}
\label{cor:d=0}
From Theorem \ref{thm: mainthm}, we have the following natural results about $\dim(\mathfrak{W})$:
\begin{align*}
\dim(\mathfrak{W}) = \begin{cases}
    \left(\frac{n}{\gcd(c,n)}\right)^r & \text{ if } d \equiv 0\pmod{n}\\
    \left(\frac{n}{\gcd(d,n)}\right)^{r-1} \cdot \frac{n}{\gcd((r-1)d, n)} & \text{ if } c\equiv 0\pmod{n}
\end{cases}
\end{align*}
\end{cor}

\begin{proof}
Recall that by Theorem \ref{thm: Frechette}, we have $\dim(\mathfrak{W}) = n^r/|S_{cochar}(c,d,r,n)|$. Then if $d \equiv 0 \pmod n$,
\begin{align*}
S_{cochar}(c,d,r,n)
&= \gcd(c-d,n)^{r-1} \gcd\left(c+(r-1)d, n, \frac{dn}{\gcd(c-d,n)}\right)\\
&= \gcd(c,n)^r.
\end{align*}
Likewise, if $c\equiv 0\equiv n$, then
\begin{align*}
S_{cochar}(c,d,r,n)
&= \gcd(-d,n)^{r-1} \gcd\left((r-1)d, n, \frac{dn}{\gcd(-d,n)}\right)\\
&= \gcd(d,n)^{r-1}\gcd((r-1)d,n).\qedhere
\end{align*}
\end{proof}

As we can see from this corollary, the parameters $c$ and $d$ play significantly different roles in influencing the structure of the Whittaker function space. In the simplest $n$-fold metaplectic cover ($c = 1, d= 0$), we see $|\widetilde{T}| = n^r$, which allowed 
Brubaker, Bump, and Buciumas to map $\mathfrak{W}$ isomorphically to a quantum module of dimension $n^r$ in \cite{BBBIce} to explain the lattice model phenomena discovered by Brubaker, Bump, Chinta, Friedberg, and Gunnells \cite{BBCFG}. In the same spirit, the second author showed in \cite{Frechette} that this connection extends quite naturally to an isomorphism for any cover coming from a diagonal matrix (i.e., $d\equiv 0$). However, incorporating the parameter $d$ adds complications, as the quantum module (which we will discuss later in Section) \emph{does not see} the factor of $\gcd\left(c+(r-1)d, n, \frac{dn}{\gcd(c-d,n)}\right)$ appearing in $\dim(\mathfrak{W})$. Thus, to understand this connection, we will need additional information about the structure of $\mathfrak{W}$.

\begin{cor}
\label{cor:c=d=0}
We have $\dim(\mathfrak{W}) = 1$ (that is, of minimum size) if and only if $c\equiv d\equiv 0\pmod{n}$.

\end{cor}

\begin{proof}
The backward direction follows from Corollary \ref{cor:d=0}. Now assume $S_{cochar}=n^r$. Since each of the $r$ factors in Theorem \ref{thm: mainthm} are factors of $n$, we must have $\done=\gcd(c-d,n) =n$ and so 
    $$S_{cochar} = n^{r-1} \gcd\left(c-(r+1)d, n, c,d\right).$$ 
So we must also have $\gcd(n,c,d) =n$, which requires that $c,d \equiv n \pmod{n}.$
\end{proof}

\begin{cor}\label{cor:cdcoprime}
We have $\dim(\mathfrak{W}) = n^r$ (that is, of maximum size) if and only if $c-d$ and $c+(r-1)d$ are coprime to $n$.
\end{cor}

\begin{proof}
It suffices to show that $S_{cochar} = 1$ if and only if $\done = \dtwo = 1$.
The backwards direction is easiest to see from Theorem \ref{thm:Main2}: if $\done= \dtwo=1$, then
\[ S_{cochar} = \frac{1}{n}\gcd(r,n)\cdot \lcm\left(\frac{n}{\gcd(r,n)},1\right) = 1.\]

For the forward direction, we use Theorem \ref{thm: mainthm}. Here, $S_{cochar} = 1$ implies both $\done^{r-1} = 1$ (and thus $\done = 1$) and $\gcd\left(\dtwo, \frac{dn}{\done}\right) = 1$. Since $\done = 1$ we then have $\gcd\left(\dtwo, dn\right) = 1$ which tells us that $\dtwo$ must be relatively prime to $n$. But $\dtwo = \gcd(c+(r-1)d,n)$ so thus $\dtwo = 1$.
\end{proof}

We will later see that both maximizing and minimizing $\mathfrak{W}$ result in very nice quantum connections.

It is also intriguing to ask when the diagonal number phenomenon developed in Section \ref{Cochar} matches the dimension precisely: that is, when is $|\Lambda_{fin}| = \done^{r-1}\dtwo$? One case in which this is true is fairly straightforward.

\begin{cor}\label{cor:coprime}
If $n$ and $r$ are relatively prime, $\dim(\mathfrak{W}) = n^r/\left(\done^{r-1}\dtwo\right)$.
\label{cor: r n rel prime}
\end{cor}

\begin{proof}
Suppose $\gcd(r,n) = 1$ and consult Theorem \ref{thm:Main2}. 
Then,
\[ S_{cochar} = \frac{\done^{r-1}\dtwo}{n}\cdot \lcm\left(n, \gcd\left(c+(r-1)d,n, \frac{dn}{\done}\right)\right)  = \done^{r-1}\dtwo\]
as the $\gcd$ above is a factor of $n$.
\end{proof}

However, the general conditions are a bit more complicated.

\begin{prop} Suppose $n = p_1^{m_1}\cdots p_j^{m_j}$,, and for each $p_i$, we have $c-d \equiv c_i p_i^{s_i}\pmod{p_i^{m_i}}$, $d \equiv d_i p_i^{t_i}\pmod{p_i^{m_i}}$, and $r\equiv r_i p_i^{\mu_i} \pmod{p_i^{m_i}}$, where $c_i, d_i,$ and $r_i$ are coprime to $p_i$. Then we have $\dim(\mathfrak{W}) = n^r/\left(\done^{r-1}\dtwo\right)$ if and only if one of the following three conditions is true for every $i$: 
\begin{itemize}
    \item $s_i < t_i + \mu_i$ and $2s_i \leq t_i+m_i$,
    \item $s_i > t_i+\mu_i$ and $s_i \leq m_i-\mu_i$, or
    \item $s_i = t_i + m_i$ and $2s_i + \tau_i \leq t_i+m_i$, where $\tau_i$ is the number of powers of $p_i$ in $c_i+r_id_i$.
\end{itemize}
\end{prop}

\begin{proof}
Using Theorem \ref{thm: mainthm}, $S_{cochar} = \done^{r-1}\dtwo$ if and only if $\gcd(\dtwo, \frac{dn}{\done}) = \dtwo$. That is, precisely when $\dtwo$ divides $\frac{dn}{\done}$. Since the left side divides $n$, it suffices to check that for every prime factor $p_i$ of $n$, the power of $p_i$ in $\dtwo$ divides that in $\frac{dn}{\done}$. Given any $p_i$, by the proof of Proposition \ref{prop: M def as primes}, we know that the power of $p_i$ on the right hand side is $t_i+m_i-s_i$. Similarly, the power of $p_i$ appearing in $c-d+rd$ is $\min\{s_i,t_i+\mu_i\} + \tau_i\cdot\delta_{s_i=t_i+\mu_i}$, where $\tau_i$ is the power of $p_i$ appearing in $c_i+r_id_i$. Thus, it suffices to determine exactly when
\begin{align}\label{ppowers75}
\min\{s_i,t_i+\mu_i\} + \tau_i\cdot\delta_{s_i=t_i+\mu_i} \leq t_i+m_i-s_i.\end{align}

To do so, we split into the same cases we used in the proof of Proposition \ref{prop: M def as primes}, based on the power of $p_i$ appearing in $\dtwo$. First, suppose $s_i < t_i + \mu_i$. Then we wind up in the first condition, because \eqref{ppowers75} is true precisely when
\[s_i \leq t_i+m_i-s_i.\]
Then, suppose $s_i > t_i+\mu_i$. Then \eqref{ppowers75} is true if and only if 
\[ \mu_i \leq m_i - s_i\]
satisfying the second condition. Lastly, suppose $s_i = t_i + \mu_i$. Then \eqref{ppowers75} is equivalent to the third condition 
\[2s_i + \tau_i \leq t_i+m_i. \qedhere\]
\end{proof}

Using the same techniques, we can also describe all the cases when $\dim(\mathfrak{W}) =(n/\done)^r$. As we will see later, the quantum module connected to $\mathfrak{W}$ has dimension $(n/\done)^r,$ so this is a necessary condition for the map to be an isomorphism.


\begin{prop}\label{prop:sizematches}
Suppose $n = p_1^{m_1}\cdots p_j^{m_j}$,, and for each $p_i$, we have $c-d \equiv c_i p_i^{s_i}\pmod{p_i^{m_i}}$, $d \equiv d_i p_i^{t_i}\pmod{p_i^{m_i}}$, and $r\equiv r_i p_i^{\mu_i} \pmod{p_i^{m_i}}$, where $c_i, d_i,$ and $r_i$ are coprime to $p_i$. Then $\dim(\mathfrak{W}) = (n/\done)^r$ if and only if for every $i$, we have $2s_i \leq m_i+t_i$ and at least one of the following conditions:
\begin{itemize}
\item $s_i< t_i+\mu_i$,
\item $s_i = t_i+\mu_i$ and $2s_i = t_i+m_i$, or
\item $s_i = t_i+ \mu_i$ and $c+(r-1)d$ contains no additional powers of $p_i$.
\end{itemize}
\end{prop}

\begin{proof}
Using Theorem \ref{thm: mainthm}, $\dim(\mathfrak{W}) = (n/\done)^r$ if and only if $\gcd(\dtwo, \frac{dn}{\done}) = \done$. Using the machinery developed in the proof of Proposition \ref{prop: M def as primes}, notice that both sides are factors of $n$, so it suffices to check that the powers of each prime $p_i$ appearing in the prime factorization of $n$ match. 

Let $n = p_1^{m_1}\cdots p_j^{m_j}$ and suppose that $c-d\equiv c_i p_i^{s_i}\pmod{p_i^{m_i}}$ and $d \equiv d_i p_i^{t_i}\pmod{p_i^{m_i}}$, where $c_i$ and $d_i$ are coprime to $p_i$. Also, note that $r \equiv r_ip_i^{\mu_i}\pmod{p_i^{m_i}}$, where $r_i$ is also coprime to $p_i$. Then the power of $p_i$ appearing in $\done$ is $s_i$. From the proof of Proposition \ref{prop: M def as primes}, recall that the power of $p_i$ appearing in $\frac{dn}{\done}$ is $t_i+m_i-s_i$ and the power of $p_i$ appearing in $c-d+rd$ is $\min\{s_i,t_i+\mu_i\} + \tau_i\cdot\delta_{s_i=t_i+\mu_i}$, where $\tau_i$ is the power of $p_i$ appearing in $c_i+r_id_i$. So $\gcd(\dtwo, \frac{dn}{\done}) = \done$ if and only if
\begin{align}\label{ppowers}
\min\{\min\{s_i,t_i+\mu_i\} + \tau_i\cdot\delta_{s_i=t_i+\mu_i},t_i+m_i-s_i\} = s_i.
\end{align}

As in the proof of Proposition \ref{prop: M def as primes}, we split into three cases. If $s_i < t_i + \mu_i$, then \eqref{ppowers} gives us \[\min\{s_i, t_i+m_i-s_i\} = s_i,\] which is true precisely when $t_i+m_i-s_i \geq s_i$, satisfying the first conditions. 

If $s_i > t_i + \mu_i$, then we have a contradiction, since the minimum in \eqref{ppowers} is already less than $s_i$, and vice versa.

Finally, if $s_i = t_i+\mu_i$, then \eqref{ppowers} is
\[\min\{ s_i + \tau_i, t_i+m_i-s_i\} = s_i,\]
which is true exactly when $2s_i = t_i+m_i$ or $\tau_i = 0$ and $s_i \leq t_i + m_i - s_i$, satisfying the second and third conditions, respectively.
\end{proof}

We have seen in this section that the two different formulations of Theorems \ref{thm: mainthm} and \ref{thm:Main2} are useful for many different purposes. While the approach used to generate Theorem \ref{thm:Main2} provides a more natural path for generalization beyond $GL_r(F)$, it would be interesting in future work to investigate whether there is an analogous approach to that used in Theorem \ref{thm: mainthm} for other reductive groups. In particular, understanding how Theorems \ref{thm: mainthm} and \ref{thm:Main2} are related for the case of $GL_r(F)$ will illuminate a path for extending this connection further. 

\section{Quantum Connections}\label{Quantum}

Finally, we marshal together results from the previous sections to investigate how the space of Whittaker functions is connected to quantum group modules, building the necessary quantum definitions along the way. 

Let $U_q(c,d,n)$ be the affine quantum group $U_q(\widehat{\mathfrak{gl}}(n/\done))$, where $q$ is the cardinality of the residue field for our nonarchimedean local field $F$. For the results of this paper, we will not need the precise definition here, so we refer the reader to Chari and Pressley \cite{CP} for the details of the construction and instead note merely a few interesting facts about $U_q(c,d,n)$. First, despite the name, $U_q(c,d,n)$ is not a group, but rather an algebra, specifically a quasitriangular Hopf algebra. That is, it is both an algebra and a coalgebra, so it comes equipped with not only multiplication and a unit map but also comultiplication, a counit, and an antipode map relating the algebra and coalgebra structures. Furthermore, this quantum group has a very nice set of modules which we can model concretely.

\begin{definition}
    For $z\in \mathbb{C}$, let $V_+(z)$ be an \emph{evaluation module}, or \emph{evaluation representation}, for $U_q(c,d,n)$. Again, we will not need the full structure of this representation for this paper, but following Kojima \cite{Kojima} as a convenient source, note that $V_+(z)$ is $n/\done$-dimensional and its basis may be parametrized by the elements of $\Z/(n/\done)\Z$.
\end{definition}

In addition, $U_q(c,d,n)$ comes with an invertible element called a \emph{universal $R$-matrix} $R \in U_q(c,d,n)\otimes U_q(c,d,n)$, which acts on tensor products of $U_q(c,d,n)$-modules. Choosing a particular pair of modules and their bases, $R$ becomes an honest-to-goodness matrix.

It is this $R$-matrix that sparked the connection between Whittaker functions and quantum groups: $R$-matrices are natural sources for solutions to Yang-Baxter equations, functional relations from statistical mechanics that that arise, among other places, in the theory of lattice models. In \cite{BBCFG}, Brubaker, Bump, Chinta, Friedberg, and Gunnells constructed a ice-type lattice model called \emph{Metaplectic Ice} which computes metaplectic Whittaker functions for the nicest cover ($c=1,d=0$, so $\done=\dtwo=1$). However, the Yang-Baxter equation for this model was unknown until Brubaker, Bump, and Buciumas identified it as a Drinfeld twist of the $R$-matrix for $U_q(\widehat{\mathfrak{gl}}(n))$ in \cite{BBBIce}. Using the lattice model as a bridge, they mapped the space of Whittaker functions on this cover isomorphically into the tensor product $V_+(z_1) \otimes \cdots \otimes V_+(z_r)$, where $z_i \in \mathbb{C}$ are the Satake parameters for the principal series representation on which the Whittaker function space $\mathfrak{W}=\mathfrak{W}^{\boldsymbol{z}}$ is built. Under this isomorphism, the action of the $R$-matrix on the components of $V_+(z_1) \otimes \cdots \otimes V_+(z_r)$ matches \emph{precisely} the action of intertwining operators on the principal series representation and thus the Whittaker function space.

Fantastically, this connection extends for \emph{any} metaplectic cover of $GL_r(F)$. In \cite{Frechette}, the second author built a generic lattice model for an arbitrary covering group, and used it to construct a map between the space of Whittaker functions and a quantum group module for the quantum group $U_q(c,d,n)=U_q(\widehat{\mathfrak{gl}}(n/\done))$. 
However, as we saw already from the formulae for $\dim(\mathfrak{W})$ and the structure theory in Section \ref{Structure}, changing the parameters $c$ and $d$ results in a significantly more complicated function space. These complications extend to the map, as the quantum space changes differently than $\mathfrak{W}$ does. In spite of this, the map prescribed by the lattice model in the fully general case is still a homomorphism and it matches exactly the actions of the R-matrix on the right side to those of the intertwining operators on the left.

Consider the tensor product $V_+(z_1) \otimes \cdots \otimes V_+(z_r)$ of quantum group evaluation modules for $U$. As a vector space, we have
\begin{align*}\dim \Bigl(V_+(z_1) \otimes \cdots \otimes V_+(z_r)\Bigr) =\left( \frac{n}{\done}\right)^{r}.\end{align*}

\noindent Note that unlike either of the formulae for $\dim(\mathfrak{W})$ in Main Theorem \ref{thm: Intromainthm}, this formula is not affected by $\dtwo$.

Now we come to the connection precisely. Using Theorem \ref{thm: Frechette}, take representatives for the cosets $\widetilde{T}/H$ from the set $(\Z/n\Z)^r$. Using Theorem \ref{thrm: background}, use these representatives to construct a basis for $\mathfrak{W}$.

\begin{thm}[Frechette \cite{Frechette}, Theorem 1.1]\label{Frechette2}
Let $\rho = (r-1,...,2,1,0)$. For $\boldsymbol{z}\in \mathbb{C}^r$, the map 
\begin{align*}
\theta_{\boldsymbol{z}}: \mathfrak{W}^{\boldsymbol{z}} &\rightarrow V_+(z_1) \otimes \cdots \otimes V_+(z_r)\\
\boldsymbol{\nu}\hspace{0.25cm} &\mapsto \hspace{0.5cm} \rho - \boldsymbol{\nu} \pmod{n/\done},
\end{align*}
where the modulus is taken in each component, is a homomorphism compatible with the actions of intertwining operators on $\mathfrak{W}$ and the R-matrix on the quantum tensor product.
\end{thm}

One of the difficulties that arose in extending from the nicest cover to generic covers is that the lattice model specifies a choice of basis for $\mathfrak{W}$ that makes this map a homomorphism, but the lattice model itself is not necessary for the proof and serves as a removable bridge between the Whittaker function space and the quantum group model. Without the lattice model, however, there is no canonical choice of basis for $\mathfrak{W}$, so we ask: when is this map well-defined regardless of the choice of representative for each coset in $\widetilde{T}/H$?

Using the structure of $\mathfrak{W}$ developed in Section \ref{Structure}, we can investigate this map more precisely, and arrive at the following theorem, which is a restatement of the first part of Main Theorem 2.

\begin{thm}\label{quantummap}
For the metaplectic cover $\widetilde{G}_{c,d,r,n}$, the map $\theta_{\boldsymbol{z}}:\mathfrak{W} \rightarrow  V_+(z_1) \otimes \cdots \otimes V_+(z_r)$ from Theorem \ref{Frechette2} is well-defined independent of choice of coset representatives for $\widetilde{T}/H$ if and only if
\[ \gcd\left(\dtwo, \frac{dn}{\done}\right) = \gcd(c,d,n).\]
\end{thm}

\begin{proof}
Using the characterization developed in Theorem \ref{thm:Main2}, $\theta_{\boldsymbol{z}}$ is well-defined if and only if all the elements in $\Lambda_{fin}$ map to the same element in the module. Using the description of Proposition \ref{prop: solutions to inhom}, write $\boldsymbol{x} = x_1 \cdot \boldsymbol{1}_r + \frac{n}{\done} (0,v_2,v_3,...,v_r)$ and 
$\boldsymbol{y}= y_1 \cdot \boldsymbol{1}_r + \frac{n}{\done} (0,v_2',v_3',...,v_r')$, for $x_1,y_1, v_i, v_i' \in \Z$ for all $i$. Thus,
\[\tzmap(\boldsymbol{x}) - \tzmap(\boldsymbol{y}) = (y_1-x_1) \cdot \boldsymbol{1}_r  \pmod{n/\done}.\]

Since $\boldsymbol{x},\boldsymbol{y}\in \Lambda_{fin}$, we have $\boldsymbol{y}- \boldsymbol{x} \in \Lambda_{fin}$ as well, so the defining cocharacter equations give more information about the possible values of $y_1-x_1$.
Using the first cocharacter equation, there exists $k\in \Z$ such that
\begin{align*}
    (c+(r-1)d) \cdot (y_1 -x_1) \equiv  \frac{dn}{\done} \cdot k \pmod{n}.
\end{align*}
Varying over all $\boldsymbol{x}, \boldsymbol{y} \in \Lambda_{fin}$, the possible values for the right hand side of this equation are precisely the integer multiples of $\gcd\left(n, \frac{dn}{\done}\right).$ Then, both sides must be a multiple of $\lcm\left(c+(r-1)d, \gcd\left(n, \frac{dn}{\done}\right)\right)$. Using the fact that $\lcm(A,B) = (A\cdot B)/\gcd(A,B)$, the possible values for $y_1-x_1$ are all the integer multiples of 
\begin{align} \label{multipleND1}\frac{\gcd\left(n, \frac{dn}{\done}\right)}{\gcd\left(c+(r-1)d,n, \frac{dn}{\done}\right)} = \frac{\frac{n}{\done}\gcd\left(\done, d\right)}{\gcd\left(\dtwo, \frac{dn}{\done}\right)} = \frac{n}{\done} \cdot \frac{\gcd\left(c,d,n\right)}{\gcd\left(\dtwo, \frac{dn}{\done}\right)}.\end{align} 

Going back to the map,  $\tzmap(\boldsymbol{x})- \tzmap(\boldsymbol{y}) = 0$ if and only if $y_1-x_1 \equiv 0 \pmod{n/\done}$. Since $\gcd(c,d,n)$ divides both $\dtwo$ and $\frac{dn}{\done}$, we have that the expression in \eqref{multipleND1} is a multiple of $\frac{n}{\done}$ if and only if $\gcd(c,d,n)=\gcd\left(\dtwo, \frac{dn}{\done}\right)$. Therefore the map $\tzmap$ is well-defined for any choice of coset representatives of $\widetilde{T}/H$ if and only if $\gcd(c,d,n)=\gcd\left(\dtwo, \frac{dn}{\done}\right)$.
\end{proof}

\begin{cor}\label{quantumIsom}
When $\mathfrak{W}$ is either maximum or minimum size, $\tzmap$ is an isomorphism.
\end{cor}

\begin{proof}
    If $\mathfrak{W}$ is of maximum size $n^r$, then by Corollary \ref{cor:cdcoprime}, we have $\done = \dtwo = 1$. Thus, $\gcd\left(\dtwo, \frac{dn}{\done}\right) = 1$, which forces $\gcd(c,d,n)=1$, so $\tzmap$ is well-defined. In this case $\widetilde{T}/H$ is parametrized by all of $(\Z/n\Z)^r$, and since $n/\done = n$, so is the quantum module. Looking at the description of $\tzmap$ in Theorem \ref{Frechette2}, we see that $\tzmap$ is an isomorphism by definition, flipping $\mathfrak{W}$ and shifting by $\rho$.
 
    If $\mathfrak{W}$ is of minimum size $1$, then Corollary \ref{cor:c=d=0} shows that $\done = \dtwo = n$. Thus, $\gcd(c,d,n) = n,$ which forces $\gcd\left(\dtwo, \frac{dn}{\done}\right) = n$ and makes $\tzmap$ well-defined. Here, $\widetilde{T}/H$ is a single element, which maps to the single element $\boldsymbol{0}$ in $(\Z/(n/\done)\Z)^r$, since $n/\done = 1$. Thus the map is vacuously an isomorphism.
\end{proof}

Note that the first case of Corollary \ref{quantumIsom} includes the nicest cover $c=1, d=0$ originally treated by \cite{BBCFG} and \cite{BBBIce}, explaining why the quantum map on $\mathfrak{W}$ for this case is an isomorphism.

Using our description of $\Lambda_{fin}$, we intend in the future to come up with a precise description of the structure of $\mathfrak{W}$ in the style of Corollary \ref{quantumIsom} for more general cases, which will allow us to characterize the precise behavior of $\tzmap$. In particular, we are interested in providing a companion to Proposition \ref{prop:sizematches} by finding a sufficient condition for all cases when $\tzmap$ is an isomorphism. Extending our methods and results for $\mathfrak{W}$ from $GL_r(F)$ to arbitrary reductive groups will then give us more information about what the quantum objects connected to other types of reductive groups should be. While some solvable lattice models for other types exist, they have not yet been linked to modules for quantum groups or other quantum algebraic objects, so we believe that investigating the dimension and description of $\mathfrak{W}$ for other types will illuminate likely candidates for broader quantum connections.

\bibliography{bibliography}{}
\bibliographystyle{amsplain}

\end{document}